\documentclass[12pt, letterpaper, reqno]{amsart}

\renewcommand{\Im}{\operatorname{Im}}
\newcommand{\sech}{\operatorname{sech}}
\newcommand{\defeq}{\stackrel{\rm{def}}{=}}

\usepackage{amssymb}
\usepackage{amsthm}
\usepackage{amsxtra}
\usepackage{graphicx}
\usepackage{color}
\usepackage{amsmath}
\usepackage{listings}
\usepackage{verbatim}
\usepackage{float}
\usepackage{cases}
\usepackage{threeparttable}
\usepackage{dcolumn}
\usepackage{multirow}
\usepackage{booktabs}

\newtheorem{theorem}{Theorem}[section]
\newtheorem{property}{Spectral Property}
\newtheorem{definition}[theorem]{Definition}
\newtheorem{proposition}{Proposition}[section]

\newtheorem{remark}[theorem]{Remark}

\numberwithin{equation}{section}

\newcommand{\ds}{\displaystyle}

\begin{document}

\title[Stable blowup in NLS in high dimensions]
{Blow-up dynamics and spectral property\\
in the $L^2$-critical nonlinear Schr\"odinger equation\\
in high dimensions}

\author[K.Yang, S. Roudenko, Y. Zhao]{Kai Yang, Svetlana Roudenko and Yanxiang Zhao}

\address{George Washington University}



\date{} 

\begin{abstract}
We study stable blow-up dynamics in the $L^2$-critical nonlinear Schr\"{o}dinger equation in high dimensions. 
First, we show that in dimensions $d=4$ to $d=12$ generic blow-up behavior confirms the {\it log-log} regime in our numerical simulations {under the radially symmetric assumption}, and asymptotic analysis, including the log-log rate and the convergence of the blow-up profiles to the rescaled ground state; this matches the description of the stable blow-up regime in 
the $2d$ cubic NLS equation.

Next, we address the question of rigorous justification of the {\it log-log} dynamics in higher dimensions ($d \geq5$), at least for the initial data with the mass slightly larger than the mass of the ground state, for which the spectral conjecture has yet to be proved,
see \cite{MR2005} and \cite{FMR2006}. We give a numerically-assisted proof of the spectral property for the dimensions from $d=5$ to $d=12$, and a modification of it in dimensions $2 \leq d \leq 12$. This, combined with previous results of Merle-Rapha\"el, proves the {\it log-log} stable blow-up dynamics in dimensions $d \leq 10$ and radially stable for $d \leq 12$.
\end{abstract}

\maketitle

\tableofcontents

\section{Introduction}

We consider the Cauchy problem of the $L^2$-critical nonlinear Schr\"{o}dinger (NLS) equation:
\begin{align}\label{NLS1}
\begin{cases}
iu_t+\Delta u + |u|^{\frac{4}{d}}u=0, \qquad (t,x) \in \mathbb{R} \times \mathbb{R}^d \\
u(x,0)=u_0 \in H^1(\mathbb{R}^d).
\end{cases}
\end{align}

The local well-posedness of the equation (\ref{NLS1}) in $H^1$ is given by Ginibre and Velo in \cite{GV1979}, see also \cite{Ca2003}. It shows that for $u_0 \in H^1(\mathbb{R}^d)$, there exists $0<T\leq \infty$ such that there is a unique solution $u(t) \in \mathbb{C}([0,T),H^1(\mathbb{R}^d))$. We say the solution exists globally in time if $T=\infty$, and the solution blows up in finite time if $T<\infty$ and $\ds \limsup\limits_{t\,\nearrow\, T} \|\nabla u(t)\|_{L^2}=\infty$.

During their lifespan, the solutions $u(x,t)$ of the Cauchy problem \eqref{NLS1} satisfy the following conservation laws
\begin{align} \label{conservation}
&\textit{Mass}: \quad  M(u(x,t)) \defeq \int |u(x,t)|^2 \, dx = \int |u_0(x)|^2 \, dx, \\ 
&\textit{Energy}:  \quad  E(u(x,t))\defeq \frac{1}{2}\int|\nabla u(x,t)|^2 \, dx -\frac{1}{2+\frac{4}{d}}\int |u(x,t)|^{2+\frac{4}{d}} \, dx = E(u_0),\\
&\textit{Momentum}:  \quad P(u(x,t)) \defeq \Im \left( \int\nabla u (x,t) \bar{u}(x,t) \, dx  \right) = \Im \left( \int \nabla u_0(x) \bar{u}_0(x) \, dx \right). 
\end{align}

Substituting $u(t,x)=e^{it}Q(x)$ into \eqref{NLS1}, we obtain a special class of periodic solutions with $Q$ solving
\begin{equation}\label{GS}
-\Delta Q +Q -|Q|^{\frac{4}{d}}Q=0.
\end{equation}
The equation \eqref{GS} exhibits many solutions, even when restricted to those that are smooth and vanishing-at-infinity. In this work we are only interested in real positive solutions. The existence and uniqueness of a real, positive, vanishing-at-infinity solution of (\ref{GS}) are obtained in \cite{BL1983} and \cite{Kwong1989}, and this solution is referred to as the ground state solution, which we also denote by $Q$. Note that the ground state solution is radially symmetric $Q=Q(r)$, moreover, it is exponentially decaying at infinity, \cite{BL1983}. For the purpose of this paper, we will simply say $Q\in H^1(\mathbb{R}^d)$. In the dimension $d=1$, the ground state solution of (\ref{GS}) is explicit
\begin{equation}\label{E:Q-1d}
Q(x)=3^{1/4} \sech^{1/2}(2x).
\end{equation}
While in dimensions $d\geq 2$, the ground state is not explicit, its properties are known, and various numerical methods (e.g., renormalization method or shooting method, see, for example, the monograph \cite{F2016} and reference therein) produce the ground state $Q$ numerically.

The importance of the ground state comes into play when one wants to understand the long time behavior of solutions. In 1983, Weinstein \cite{W1983} showed that solutions exist globally in time if $\|u_0\|_{L^2}<\|Q\|_{L^2}$. By convexity arguments on a finite variance, it is known that solutions blow up in finite time if $E[u_0]<0$, see \cite{VPT1971}, \cite{Z1972} or \cite{G1977}. Thus, solutions with $\|u_0\|_{L^2} \geq \|Q\|_{L^2}$ may blow up in finite time. Moreover, the minimal mass blow-up solutions happen exactly at the threshold $\|u_0\|_{L^2} = \|Q\|_{L^2}$, and the classification of all minimal mass blow-up solutions was done by Merle in \cite{M-CPAM} (radial) and \cite{M-Duke} (general case). We note that all minimal mass blow-up solutions are unstable. In this paper we are interested in stable blow-up dynamics, thus, we consider the initial data with the mass above the mass of the ground state $Q$.

Investigations of the stable blow-up dynamics for the NLS equation (mainly in the two dimensional case, i.e., cubic nonlinearity) go back to 1970's (for example, \cite{BZS1975}, \cite{ZS1976}, \cite{SZ}, \cite{SZ2}).
In 1986, McLaughlin, Papanicolaou, Sulem and Sulem in \cite{MPSS1986} (see also  \cite{SS1999}) showed that the rate of stable blow-up should be coming from scaling and equivalent to $(T-t)^{-\frac{1}{2}}$ by applying the dynamic rescaling method. Previously, Talanov and collaborators (1978), Wood (1984) and Rypdal and Rasmussen (1986) suggested a form of $(|\ln(T-t)|/(T-t))^{\frac{1}{2}}$ from a different approach, see \cite{VPT1978}, \cite{W1984}, \cite{RR1986}. Then, Landman, Papanicolaou, Sulem, and Sulem in \cite{LPSS1988} (also  LeMesurier, Papanicolaou, Sulem and Sulem in \cite{LePSS1988}) as well as Fraiman \cite{F85},  and others (see \cite{LLPSS1989}, \cite{DNPZ1992}, \cite{KSZ1991}, \cite{M1993}; for a more complete reference list on this subject refer to \cite{SS1999})
concluded from asymptotic analysis, incorporating numerical simulations, that the rate of blow-up is of the form $\ds \left( \frac{\ln|\ln(T-t)|}{(T-t)} \right)^{\frac{1}{2}}$, which is now referred to as the ``log-log" law. Later in 1990, Landman, Papanicolaou, Sulem, Sulem and Wang in \cite{LPSSW1990} concluded that this log-log rate is a stable blow-up rate in 2d by simulating the blow-up dynamics without radial symmetry assumption. Note that numerical tracking of the log-log correction is extremely difficult as it is reached only at exceedingly huge focusing levels (of orders $\sim$ 200), for example, see \cite{FP1999}, and thus, asymptotic analysis was used in \cite{LPSS1988, LePSS1988} to obtain the log-log correction. An alternative approach was developed by Akrivis, Dougalis, Karakashian and McKinney in 1993 and 1997, see \cite{ADKM1993}, \cite{ADKM2003}, where very accurate computations of the solution were performed using an adaptive Galerkin finite element method. The computed solution was then tested against several theoretical predictions for the functional form of the correction, and the log-log law was selected for the far asymptotic regime as the best fitting. Recently, there have been (very refined and to exorbitantly high focusing levels) numerical studies beyond the leading order of the stable blow-up dynamics, see Lushnikov, Dyachenko and Vladimirova \cite{LDV2013}, where the authors obtained higher order correction terms in the log-log blow-up rate.

While it was important to investigate the stable blow-up dynamics numerically and heuristically, the rigorous analytical description started appearing only at the turn of this century. Galina Perelman in 2001, see \cite{GP2001} (also \cite{GP1999}), made a rigorous construction of the ``log-log" blow-up solutions in the 1d case (i.e., for the quintic NLS equation $iu_t +u_{xx}+|u|^4u=0$). As it was the first analytical proof of the ``log-log" blow-up, it opened possibilities to study blow-up solutions rigorously. In the series of papers \cite{MR2005}, \cite{MR2003}, \cite{MR2004},  \cite{R2005}, \cite{MR2006}, 
Merle and Rapha\"{e}l did a systematic study of the generic blow-up, considering initial data slightly above the ground state mass and with negative energy, and obtained a detailed description of the dynamics, in particular, the convergence of the blow-up profiles to the self-similarly rescaled $Q$ and the blow-up rate $\ds \left( \frac{\ln|\ln(T-t)|}{(T-t)} \right)^{\frac{1}{2}}$. This description was proved for any dimension, provided the so-called ``spectral property" holds true (see the actual statement below). This spectral property was proved rigorously in the one dimensional case in \cite{MR2005}, since the ground state $Q$ is explicit in 1d  \eqref{E:Q-1d} and the spectral property could be deduced from some known properties of the second-order differential operators \cite{T}. For the dimensions $d=2, 3, 4$, Fibich, Merle and Rapha\"{e}l gave a numerically-assisted proof of the spectral property, using the numerical representation of $Q$, see \cite{FMR2006} and \cite{FMR2008}. For 
$d\geq 6$, 
the authors in \cite{FMR2006} indicated that the spectral property is indecisive (via their approach), furthermore, there was no direct numerical simulation of the blow-up solutions in higher dimensions, even under the radial symmetry assumption (all prior $L^2$-critical NLS numerical simulations were in 2d), as it could have been a computationally difficult task at that time. Investigating stable blow-up regime in higher dimensions is the goal of this work.

In this paper, we first show that in higher dimensions, $4 \leq d \leq 12$, a generic self-similar blow-up dynamics is also described by the ``log-log" law, i.e., our numerical simulations show that the blow-up profiles converge to the self-similar ground state, furthermore, we obtain the log-log blow-up rate via derivation as in \cite[Section 8]{SS1999} and matching refinement adapted from \cite{ADKM2003}. As we only numerically simulate the radial case, we conclude in this first part that this log-log blow-up dynamics is radially stable. 
Secondly, we give a proof of the spectral property for $d=5, ..., 12$, which completes the rigorous proof of the stable blow-up dynamics for the dimensions $d \leq 12$ for the initial data with negative energy and mass slightly above the mass of the ground state.

For consistency, in what follows, we adopt the notation from \cite{FMR2006} and \cite{SS1999}: denote by $B_{\alpha}$ a neighborhood of $H^1$ functions with the $L^2$ norm slightly above the mass of the ground state $Q$, i.e., $B_{\alpha} = \left\{f \in H^1 (\mathbb R^d)\, : ~ \int Q^2 < \int |f|^2 < \int Q^2 + \alpha \right\}$; the variable $x=(x_i)_{1\leq i\leq d}$ is used as the space variable, $r=|x|$ stands for the radial coordinate and the radial derivative $\partial_r f(r)$ we typically write as $f_r$; we also denote $H^1_r$ as the space of radial functions in $H^1$, and $\langle \; , \; \rangle$ as the standard inner product on $L^2$.
We define the scaling symmetry generator acting on $Q$ by
\begin{equation}\label{E:LambdaQ}
Q_1 =\frac{d}{2}\, Q+x\cdot \nabla Q, \quad \text{and also define} \quad
Q_2 =\frac{d}{2}\, Q_1 + x\cdot \nabla Q_1.
\end{equation}

\begin{property}[\cite{FMR2006}, \cite{MS2011}] \label{D:SP}
Let $d\geq 1$. Consider the two real Schr\"{o}dinger operators
\begin{align}
L_1=-\Delta + V_1 \quad \mathrm{and} \quad L_2=-\Delta + V_2,
\end{align}
where
\begin{align}
& V_1(r)=\frac{2}{d}\left( \frac{4}{d}+1 \right)Q^{\frac{4}{d}-1}rQ_r,\\
& V_2(r)=\frac{2}{d}Q^{\frac{4}{d}-1}rQ_r,
\end{align}
and the real-valued quadratic form for $u=f+ig \in H^1(\mathbb{R}^d)$ is defined as
\begin{align}
B(u,u)&=B_1(f,f)+B_2(g,g)\\
      &=\langle L_1f,f \rangle +\langle L_2g,g \rangle.
\end{align}
Then there exists a universal constant $\delta_0>0$ such that for any $u \in H^1(\mathbb{R}^d)$, if the following orthogonal conditions hold
\begin{align}\label{orthogonal conditions1}
 \langle f, Q \rangle=\langle f, Q_1 \rangle=\langle f, x_i \, Q \rangle_{1\leq i\leq d}=0,
\end{align}
\begin{align}\label{orthogonal conditions2}
\langle Q_1, g \rangle=\langle Q_2, g \rangle=\langle \partial_{x_i}Q,g \rangle_{1\leq i\leq d}=0,
\end{align}
then $B(u,u)>0$, or more precisely,
\begin{align}\label{E:Bpositive}
B(u,u)\geq \delta_0 \left( \int|\nabla u|^2 +|u|^2e^{-|x|}   \right).
\end{align}
\end{property}

We note that the commutator-type formulas hold for the operators $L_1$ and $L_2$, which was observed in \cite{MR2005},
$$
L_1f= \frac{1}{2}\left[ L_+(f_1) - (L_+f)_1 \right] \quad \text{and} \quad
L_2g=\frac{1}{2}\left[ L_-(g_1) - (L_-g)_1 \right],
$$
where $L_+$ and $L_-$ are the standard linearized operators obtained from the linearization of solutions around the ground state $Q$:
\begin{align}
& L_+=-\Delta +1 -\left( 1+\frac{4}{d} \right)Q^{\frac{4}{d}},\\
& L_-=-\Delta +1 -Q^{\frac{4}{d}}.
\end{align}
We mention that the choice of the orthogonal conditions (\ref{orthogonal conditions1}) and (\ref{orthogonal conditions2}) in the above Spectral Property 1 comes from the generalized null space of the matrix operator $ \mathcal{L}= \left[ \begin{matrix}
0 \,&L_- \\
-L_+ \, &0
\end{matrix}\right]$ (see, e.g., Weinstein \cite{W1986}, Buslaev-Perelman \cite{BP}).
\smallskip

Our first result is the following:
\begin{theorem}\label{T:SP}
The Spectral Property 1 holds true in dimensions from $d=5$ to $d=10$.
The Spectral Property 1 holds true in the radial case for dimensions $d=11$ and $d=12$.
\end{theorem}

We next modify the statement of the Spectral Property 1. 
\begin{property}\label{D:SP2}
Let $d \geq 2$ and assume the same set up as in the Spectral Property 1, except
replace the orthogonal conditions \eqref{orthogonal conditions2} by 
\begin{align}\label{E:ortho-Q}
\langle Q, g \rangle = \langle \partial_{x_i}Q,g \rangle_{1\leq i\leq d}=0. 
\end{align}
Then $B(u,u)>0$, or \eqref{E:Bpositive} is true (i.e., the same conclusion as in the Spectral Property 1 holds).
\end{property}

Our second result is about the modified Spectral Property.
\begin{theorem}\label{T:SP2}
The Spectral Property 2 holds true in dimensions from $d=2$ to $d=10$.
The Spectral Property 2 holds true in the radial case for dimensions $d=11$ and $d=12$.
\end{theorem}

The Spectral Property 1 added to the work of Merle-Rapha{\"e}l \cite{MR2005}, \cite{MR2003}, \cite{MR2004},  \cite{MR2006}, and Fibich-Merle-Rapha{\"e}l \cite{FMR2006},  
now fully completes the proof of the stable blow-up dynamics in dimensions $d \leq 10$, and radially stable in $d \leq 12$ with the description of various dynamical features (profile, rate, phase, control of the remainder, etc; we note that the bounded control in the external region is given in \cite{HR2012}). For completeness of this introduction, we provide a concise statement of the stable blow-up dynamics  below. 

\begin{theorem}[Stable blow-up dynamics of (\ref{NLS1})]\label{T:Main}
Assume the spectral property holds true, which is now proved for $d \leq 10$ and in the radial case for $d = 11, 12$.

There exist universal constants $\alpha>0$ and $C>0$ such that the following holds true. For $u_0 \in B_{\alpha}$, let $u(t)$ be the corresponding $H^1$ solution to (\ref{NLS1}) on $[0,T)$, the interval of the maximal in forward time existence of $u$.
\begin{itemize}
\item[(i)]
Description of the singularity: Assume that $u(t)$ blows up in finite time, i.e., $0<T<\infty$. Then there exist parameters $(\lambda(t),x(t),\gamma(t)) \in  \mathbb{R} \times \mathbb{R}^d \times \mathbb{R}$ and asymptotic profile $u^* \in L^2$ such that
$$
u(t)-\frac{1}{\lambda(t)^{\frac{d}{2}}} Q\left(\frac{x-x(t)}{\lambda(t)} \right)e^{i \gamma(t)} \rightarrow u^* \quad in \, L^2(\mathbb R^d) \quad as \: ~ t \rightarrow T.
$$
Moreover, the blow-up point is finite in the sense that
$$
x(t) \rightarrow x(T) \in \mathbb{R}^d \quad as \: ~ t\rightarrow T.
$$

\item[(ii)]
Estimates on the blow-up speed: for $t$ close enough to T, either
\begin{align}\label{loglog}
\lim_{t\rightarrow T} \frac{|\nabla u(t)|_{L^2}}{|\nabla Q|_{L^2}}\left( \frac{T-t}{\ln |\ln (T-t)|}\right)^{\frac{1}{2}}=\frac{1}{\sqrt{2\pi}},
\end{align}
or
$$
|\nabla u(t)|_{L^2} \geq \frac{C(u_0)}{T-t}.
$$
The equation (\ref{loglog}) is referred to as the ``log-log" blow-up rate.

\item[(iii)]
Sufficient condition for ``log-log" blow-up\footnote{A more general description is $\ds E(u_0)< \frac12 \, \frac{[P(u_0)]^2}{M(u_0)}$, see \cite{MR2005}, \cite{MR2003}.}: If $E(u_0) < 0$ and $\int Q^2 < \int |u_0|^2 < \int Q^2 +\alpha$, then $u(t)$ blows up in finite time with the ``log-log" rate (\ref{loglog}). More generally, the set of initial data $u_0 \in B_{\alpha}$ such that the corresponding solution $u(t)$ to (\ref{NLS1}) blows up in finite time $0<t< \infty$ with the ``log-log" speed (\ref{loglog}) is open in $H^1$.
\end{itemize}
\end{theorem}
\begin{remark}
Our numerical simulations show that for the initial data $u_0$ of the Gaussian-type, $u_0 \sim  Ae^{-r^2}$, or of the ground state-type, $u_0 \sim A\, e^{-r}$, the parameter $\alpha$ in Theorem \ref{T:Main}(iii) can be taken very large, i.e., $\alpha=\infty$.
\end{remark}

This paper consists of two parts. In the first part, Section \ref{S:2}, we show the results from the direct numerical simulations of solutions in the $L^2$-critical NLS equation in dimensions $4 \leq d \leq 12$ by the dynamic rescaling method for the radially symmetric data. This shows that the blow-up rate is $(T-t)^{-\frac{1}{2}}$, possibly with some correction terms. Applying the arguments from \cite{LPSS1988} and \cite{LePSS1988}, we show that the correction term indeed exists, which can be obtained via asymptotic analysis $q(t) \approx \left((2\pi)/\ln \ln \frac{1}{|T-t|}\right)^{\frac{1}{2}}$, and to further confirm it, we apply the approach from \cite{ADKM2003} to fit the solution with various functional forms.
This leads us to the conclusion that at least for the radial data in the dimensions $d=4, ..., 12$, the ``log-log" blow-up dynamics is generic and stable, see Section \ref{S:Numerics} for details.
In the second part of the paper, Section \ref{S:SP}, we revisit the Spectral Property~1 and give a numerically-assisted proof of it for dimensions up to $d=10$ (general case), and for the radially symmetric case for $d=11$ and $d=12$; we then also establish the Spectral Property~2. In Appendices, we provide the rest of the numerical simulations in dimensions $d = 6,..., 12$ (the simulations for $d=4, 5$ are in Section \ref{S:Numerics}); we discuss the artificial boundary conditions; we describe a new approach to compute the potentials $V_{1,2}$ in high dimensions; and also provide comparison with previous results from \cite{FMR2006} of the Spectral Property 1 for $d=2,3,4$.
\smallskip

\textbf{Acknowledgements}: The authors would like to thank Gideon Simpson for his visit to GWU and posing some of the questions addressed in this paper as well as for helpful discussion on this topic. The authors are grateful for the fruitful discussions, the initial guidance for numerics and comments on the initial draft of the paper to Pavel Lushnikov, who also visited GWU in April 2017, and whose trip was partially supported by GWIMS. S.R. would like to thank Pierre Rapha\"{e}l for bringing to her attention \cite{FMR2008} and for posing the question about the orthogonality condition \eqref{E:ortho-Q}. K.Y. and Y.Z. would like to acknowledge Yongyong Cai, who hosted their visit to the Computational Science Research Center (CSRC) in Beijing during Summer 2017. S.R. was partially supported by the NSF CAREER grant DMS-1151618 as well as part of the K.Y.'s graduate research fellowship to work on this project came from the above grant. Y.Z. was partially supported by the Simons Foundation through Grant No. 357963.

\section{Numerical simulations of the solutions} \label{S:2}

In this section, we first review the dynamic rescaling method which is applied in our numerical simulations. Then, we present our numerical results. These numerical findings, combined with the analysis in \cite{LPSS1988} and \cite{LePSS1988}, adapted to our setting, shows the ``log-log" blow-up rate.

\subsection{Dynamic rescaling method}\label{S:DRMethod} 

We use the dynamic rescaling method (see \cite{MPSS1986} and \cite{SS1999}) for simulating the blow-up solutions of the equation (\ref{NLS1}). The key idea of this method is to appropriately rescale the equation (\ref{NLS1}) in both time and space variables. Then solutions for the rescaled equation exist globally in time. {We restrict our numerical simulations under the radially symmetric assumption}. Letting $\sigma=\frac{2}{d}$, where $d$ is the dimension, we set
\begin{align}\label{rescaled initial}
u(r,t)=\frac{1}{L(t)^{\frac{1}{\sigma}}}v(\xi, \tau), \quad \mbox{where} \quad \xi=\frac{r}{L(t)}, \quad \tau=\int_0^t\frac{ds}{L^2(s)}.
\end{align}
{Given a prescribed value $L(0)$ (e.g., $L(0)=1$), we obtain the initial condition $v_0(\xi)={L(0)^{\frac{1}{\sigma}}}u_0(r)$.}
Then the equation (\ref{NLS1}) becomes
\begin{align}\label{DRNLS}
iv_{\tau}+ia(\tau)\left(\xi v_{\xi}+\frac{v}{\sigma}\right)+\Delta_{\xi} v + |v|^{2\sigma}v=0,
\end{align}
{with boundary conditions 
\begin{align}\label{DRNLS_BC}
v_{\xi}(0)=0, \qquad v(\infty)=0,
\end{align}}
where
\begin{align}\label{E:a}
a=-L\frac{dL}{dt}=-\frac{d \ln L}{d\tau}.
\end{align}
There are various choices for tracking the parameter $L(t)$. The first choice is the following: since we want to bound  $\|\nabla u(t)\|_{L^2}$ as $t\rightarrow T$, we choose the parameter $L(t)$ such that the value of $\|\nabla_{\xi} v(\tau)\|_{L^2}$ in the rescaled equation (\ref{DRNLS}) remains constant in time, i.e.,
\begin{align}
L(t)=\left(\frac{\|\nabla_{\xi} v_0 \|_{L^2}}{\|\nabla u(t)\|_{L^2}} \right)^{2/p}, \qquad p=2+\frac{2}{\sigma}-d.
\end{align}
Then, $a(\tau)$ is, consequently,
\begin{align}
a(\tau)= -\frac{2}{p \|v_0 \|_{L^2}^2} \int_0^{\infty} |v|^{2\sigma} \mathrm{Im}\left(\bar{v} \Delta v \right) \xi^{d-1} d \xi.
\end{align}

An alternative choice for the $L(t)$ is to fix the $L^{\infty}$ norm of $v$ such that $\|v\|_{L^\infty}$ remains constant (see \cite{LePSS1988}, \cite{KSZ1991} for such a choice), thus, set
\begin{align}
L(t)=\left(\frac{v_0(0)}{\|u(t)\|_{L^\infty}}\right)^{\sigma},
\end{align}
and hence, $a(\tau)$ rewrites as
\begin{align}
a(\tau)= -\frac{\sigma}{|v_0(0)|^2}\Im\left( \bar{v} \Delta_{\xi} v \right)\vert_{ \left(0,\tau \right)}.
\end{align}
The above two choices are typical rescalings, for example, see \cite{MPSS1986} and \cite{SS1999}.

To discuss our numerical scheme, we note that the equation (\ref{DRNLS}) is of the form
\begin{align} \label{DRNLSF}
iv_{\tau} +\Delta_{\xi} v +\mathcal{N}(v)=0, \quad \tau \in [0, \infty), \quad \xi \in [0, \infty),
\end{align}
{where 
\begin{align*}
\mathcal{N}(v)=ia(\tau)\left(\xi v_{\xi}+\frac{v}{\sigma}\right)+|v|^{2\sigma}v.
\end{align*}}


Before we discretize the space variable $\xi$, we need to map the spatial domain $[0, \infty)$  onto $[-1,1)$, since $\xi \in [0,\infty)$. One possible way to do that is to set $\xi = \kappa \, \frac{1+z}{1-z}$, where $\kappa$ is a constant indicating the half number of the collocation points assigned on the interval $[0,\kappa]$, and $z$ is the Chebyshev Gauss-Lobatto collocation points on the interval $[-1,1]$, for instance see \cite{STT2015} and \cite{T2001}. {Because of the homogeneous Dirichlet boundary condition $v(\infty)=0$ on the right, we remove the last Chebyshev point, and, consequently, delete the last row and the last column of the matrix $\mathbf{M}$ in \eqref{M evolution}. This is similar with the singular behavior at $x \rightarrow \infty$ in \cite{BK2015} in which the authors impose No boundary condition at $x \rightarrow \infty$.} 
The Laplacian operator can be discretized by Chebyshev-Gauss-Lobatto differentiation (for details refer to \cite{STT2015} and \cite{T2001}). We denote the discretized Laplacian with $N+1$ collocation points by ${\Delta}_N$. {Note that
$$\Delta_{\xi} v=v_{\xi \xi}+\dfrac{d-1}{\xi}v_{\xi}.$$
The singularity at $\xi=0$ is eliminated by applying the L'H\^ospital's rule:
$$\lim_{\xi \rightarrow 0} \dfrac{d-1}{\xi}v_{\xi} =(d-1)v_{\xi \xi}. $$}

We use the following notation for $v$: let $v^{(m)}$ be the discretized $v$ at the time $\tau=m \cdot \delta \tau$, where $\delta \tau$ is the time step and $m$ is the iteration. The time evolution of (\ref{DRNLSF}) can be approximated by the second order Crank-Nicolson-Adam-Bashforth method:
\begin{align}\label{AB}
 i\frac{v^{(m+1)}-v^{(m)}}{\delta \tau}+\frac{1}{2}\left(\Delta_N v^{(m+1)}+{\Delta}_N v^{(m)} \right)+\frac{1}{2}\left( 3\mathcal{N}(v^{(m)})-\mathcal{N}(v^{(m-1)})\right)=0.
\end{align}
We reorganize (\ref{AB}) as
\begin{align}
\left(\frac{i}{\delta \tau}+\frac{1}{2}{\Delta}_N\right)v^{(m+1)}=\left( \frac{i}{\delta \tau}-\frac{1}{2} {\Delta}_N \right)v^{(m)}- \frac{1}{2}\left( 3\mathcal{N}(v^{(m)})-\mathcal{N}(v^{(m-1)})\right),
\end{align}
which is equivalent to
\begin{align}\label{M evolution}
\mathbf{M}v^{(m+1)}=F(v^{(m)},v^{(m-1)}).
\end{align}
Therefore, the time step is updated by
\begin{align}
v^{(m+1)}=\mathbf{M}^{-1}F(v^{(m)},v^{(m-1)}).
\end{align}
The inverse of the matrix $\mathbf{M}$ can be calculated only once in the beginning, since $\mathbf{M}=\left(\frac{i}{d\tau}+\frac{1}{2}{\Delta}_N \right)$, which stays the same.

{
The boundary conditions \eqref{DRNLS_BC} are imposed similar to \cite{MPSS1986}, \cite{T2001} and \cite{STT2015} as follows:
For the homogeneous Neumann boundary condition ($v_{\xi}(0)=0$) on the left, we substitute the first row of the matrix $\mathbf{M}$ by the first row of the first order Chebyshev differential matrix, and change the first element of the vector $F$ to $0$. Because of the homogeneous Dirichlet boundary condition $v(\infty)=0$ on the right, we delete the last row and column of $\mathbf{M}$ as well as the last element of the vector $F$.}

An alternative method is to use predictor-corrector method similar to \cite{F2016},
\begin{align}\label{predictor}
i\frac{v^{(m+1)}_{\text{pred}}-v^{(m)}}{\delta \tau}+\frac{1}{2}\left({\Delta}_N v^{(m+1)}_{\text{pred}}+{\Delta}_N v^{(m)} \right)+\frac{1}{2}\left( 3\mathcal{N}(v^{(m)})-\mathcal{N}(v^{m-1})\right)=0, \,\, (\text{P})
\end{align}
\begin{align}\label{corrector}
i\frac{v^{(m+1)}-v^{(m)}}{\delta \tau}+\frac{1}{2}\left({\Delta}_N v^{(m+1)}+{\Delta}_N v^{(m)} \right)+\frac{1}{2}\left( \mathcal{N}(v^{(m+1)}_{\text{pred}})+\mathcal{N}(v^{(m-1)})\right)=0. \,\, (\text{C})
\end{align}

The above two schemes lead to the similar results. Numerical test suggests that the scheme \eqref{predictor} and \eqref{corrector} is slightly more accurate than the scheme \eqref{AB}, though it is still second order in time and doubles the computational time. We mainly use the ``predictor-corrector" scheme \eqref{predictor} and \eqref{corrector} in our simulation.

{Note that it is sufficient to use $N=256$ collocation points. For any sufficiently smooth function $f$ over $[-1,1]$, $f(x)$ can be expanded by the Chebyshev polynomials
\begin{align}\label{f cheby expansion}
f(x)=\sum_{n=0}^{\infty} c_n T_n(x), \quad c_n \rightarrow \infty \,\,\mathrm{as} \,\, n\rightarrow \infty.
\end{align} 
Therefore, we choose a number $N_0$ such that the coefficients $c_n$ are sufficiently small (generally reaches the machine's accuracy $10^{-16}$) for $n \geq N_0$. Figure \ref{fig. cheby_coe} shows the Chebyshev coefficients $|c_n|$ of the numerical solution $v$ in the case $d=4$ at different times $\tau=0$ and $\tau=400$. One can see that $|c_n|$ reaches the order $10^{-12}$ around $N_0=128$, and reaches the machine's accuracy around $N_0=170$. Hence, it is sufficient to choose any $N>N_0$. Since the Chebyshev transform can be processed by the fast Fourier transform (fft) (e.g., \cite{T2001}, \cite{STT2015} and \cite{BK2015}), we choose $N=256$, a power of $2$. Similar study for the decaying behavior of the Chebyshev coefficients $c_n$ for multidomain spectral method for Schr\"odinger equation can be found in \cite{BK2015}.

\begin{figure}[ht]
\includegraphics[width=0.45\textwidth]{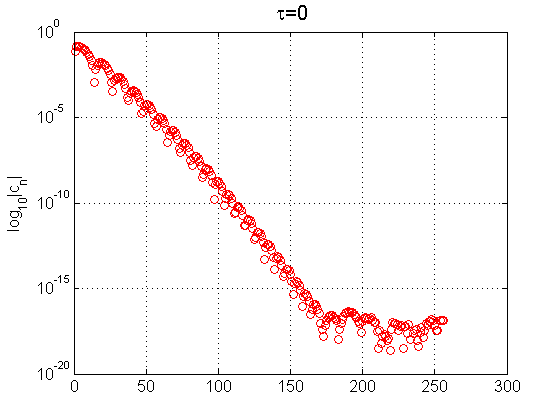}
\includegraphics[width=0.45\textwidth]{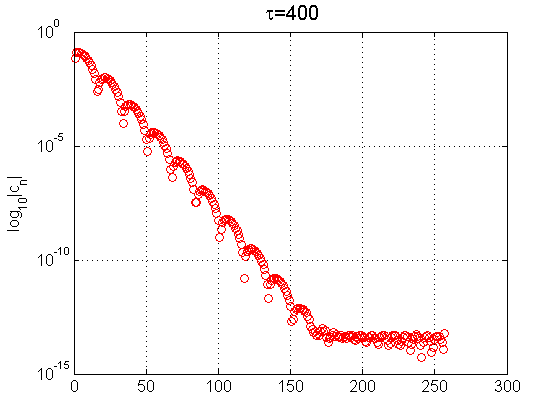}
\caption{The Chebyshev coefficients for the solution $v$ in the case $d=4$ at time $\tau=0$ (left) and $\tau=400$ (right). One can see the value $|c_n|$ reaches the machine's accuracy around $N_0=170$.}
\label{fig. cheby_coe} 
\end{figure}

 }

We choose $\kappa=256$ and $d\tau=2\times10^{-3}$. We set $v_0(0)=1$ and we choose to fix the value $\|v(\tau)\|_{\infty}\equiv 1$ in time $\tau$. Then the parameter $L(t)$ becomes $L(t)=\left( 1/ \|u(t)\|_{L^\infty} \right)^{\sigma}$, see \cite{LePSS1988}. We decided not to fix the value of $\|\nabla v(\tau)\|_{L^2}$ to be constant in high dimensions since it involves integrating $\int_0^{\infty} ... \xi^{d-1} d\xi$, which is large when the dimension $d$ increases. These two different scalings actually lead to the same slope of $L(t)$ on the $\log$ scale, i.e., $ \left(\log(L)/\log(T-t)\right) \approx \frac{1}{2}$, and also to the same decay property of $a(\tau)$, when we test the cases $d=4$ and $d=5$. Here, the term $\int_0^{\infty} ... \xi^{d-1} d\xi$ can be reduced to the order $\int_{-1}^{1-\delta_0} ... (1-z)^{3.5-d} dz$ by calculating the qunatity $a(\tau)$ from the variable $z$, where $\delta_0=1-z_{N}$ and $z_N$ is the second last discretized Chebyshev Gauss-Lobatto collocation point, i.e.,
\begin{align*}
a(\tau)&= -\frac{2}{p \|v_0 \|_{L^2}^2} \int_0^{\infty} |v|^{2\sigma} \mathrm{Im}\left(\bar{v} \Delta v \right) \xi^{d-1} d \xi \\
&=-\frac{2}{p \|v_0 \|_{L^2}^2} \int_{-1}^{1-\delta_0} |v|^{2\sigma} \mathrm{Im}( \alpha^{\frac{d}{2}-1} (1-z)^{3.5-d}(1+z)^{d-0.5}v_{zz}-\\
&{\alpha}^{d-2}(1-z)^{2.5-d}(1+z)^{d-0.5}v_z+\\
&(d-1) \alpha^{d-2}(1+z)^{d-1.5}(1-z)^{2.5-d}v_z ) \frac{1}{\sqrt{(1+z)(1-z)}}dz.
\end{align*}

{The second time step $v^{(1)}$ can be obtained by the standard second order explicit Runge-Kutta method (RK2). We stopped our simulations at the dimension $d=12$, see details and discussion on this in Section \ref{S:SP}.}

\subsection{Numerical Results}\label{S:Numerics} 
\subsubsection{Initial Data}
As in \cite{MPSS1986}, we use the Gaussian-type data $u_0(r)=Ae^{-r^2}$ as the initial data\footnote{We also ran simulations with the ground state-type of initial data $u_0 \sim A e^{-r}$, and the results were similar. For brevity, we only include the discussion on the Gaussian-type data.}, which leads to the self-similar blow-up solutions concentrated at the origin (Towerne profiles). 
We actually work with data $u_0=A_0 \, e^{-\frac{r^2}{d}}$. The reason for not taking $u_0=A\, e^{-r^2}$ as the initial data is that the amplitude $A_0$ becomes very large in higher dimensions; while the normalization term $\frac{r^2}{d}$ allows us to keep $A_0$ reasonably small.
The typical initial data is given in Table \ref{Initial}.
\begin{table}[ht]
\begin{tabular}{|c|l|l|c|c|l|}
\hline
$d$&$\|Q\|_2^2$&$\|e^{-\frac{r^2}{d}}\|_2^2$&$\tilde{A}$&$A_0$&$\|u_0\|_2^2$ \\
\hline
$4$& $20.7129$&$2$&$3.2181$&$5$&$50$\\
\hline
$5$&$112.6131$&$6.5683$&$4.1406$&$6$&$ 236.46$\\
\hline
$6$&$765.0696$&$27$&$5.3231$&$8$&$1728$\\
\hline
$7$&$6236.3848$&$133.2859$&$6.8403$&$10$&$13329$\\
\hline
$8$&$59304.81$&$768$&$8.7875$&$15$&$172800$\\
\hline
$9$&$644519.4793$&$5059.0686$&$11.2871$&$20$&$ 2.0236e6$\\
\hline
$10$&$7880266.4892$&$37500$&$14.4962$&$25$&$ 2.3437e7$\\
\hline
$11$&$107056593.2682$&$308902.5995$&$18.6164$&$30$&$ 2.7801e8$\\
\hline
$12$&$1586849773.5085$&$2799360$&$23.8089$&$40$&$4.4790e9$\\
\hline
\end{tabular}
\linebreak
\linebreak
\caption{Details about the initial data used in the dimension $d$, $u_0=A_0e^{-\frac{r^2}{d}}$, also the threshold amplitude for blow-up $\tilde{A}$, and the $L^2$ norms of the ground state $Q$, $e^{-\frac{r^2}{d}}$ and $u_0$.}
\label{Initial}
\end{table}

This table lists the mass of the ground state $Q$, the mass of $e^{-\frac{r^2}{d}}$, and we also list the threshold of the amplitude $\tilde{A}$ for the blow-up v.s. global existence. We list an example of the amplitude $A_0$ that we use in our simulations to be specific, though any amplitude $A_0>\tilde{A}$ gives the same result.

Our next Table \ref{Consistency} shows how the quantity $\|v\|_{L^{\infty}_{\xi}}$ is conserved in the rescaled time $\tau$, i.e., we track the error 
\begin{equation}\label{E:E-error}
\mathcal{E}=\max\limits_{\tau} (\|v(\tau)\|_{L^{\infty}_{\xi}})-\min\limits_{\tau} (\|v(\tau)\|_{L^{\infty}_{\xi}}).
\end{equation}
\begin{table}[ht]
\begin{tabular}{|c|c|c|c|c|c|c|c|c|c|}
\hline
$d$&$4$&$5$&$6$&$7$&$8$&$9$&$10$&$11$&$12$\\
\hline
$\mathcal{E}$&$6e-9$&$5e-10$&$1e-9$&$1e-9$&$1e-9$&$1e-9$&$2e-9$&$3e-10$&$9e-10$\\
\hline
\end{tabular}
\linebreak
\linebreak
\caption{The error $\mathcal{E}$ from \eqref{E:E-error} on the conserved quantity $\|v(\tau)\|_{L^{\infty}_{\xi}} \equiv 1$ in $\tau$ by using the predictor-corrector method (\ref{predictor}) and (\ref{corrector}) with $\delta \tau=2\times 10^{-3}$.}
\label{Consistency}
\end{table}

We comment that one should avoid choosing initial data too flat around the origin, for example, something like super-Gaussian data $u=A_0e^{-r^4}$, since this may lead to the collapsing rings instead of the Towerne profiles which we are trying to track here, for example, see \cite{HRP2010}, \cite{HR2011}, \cite{HR2007}, \cite{F2016}.

\subsubsection{Blow-up rate}
{We discuss our numerical results in the case $d=4$, the results for the dimensions $d = 5, ..., 12$ are similar, we omit them here, and refer the reader to \cite{Kai-Thesis}. }
We first observe that
\begin{equation}\label{R:half}
L(t)\sim \sqrt{T-t},
\end{equation}
see Figures \ref{fig:1.1}. {
The slope of $L(t)$ in the log scale is shown on the left, where we plotted $\log(T-t)$ vs. $\log(L)$; it gives a straight line with the slope approximately $\frac12$ (the slope is $0.50301$ in Figure \ref{fig:1.1} (left). This confirms the square root in \eqref{R:half}. We remark that for $d=2$ and $d=3$ similar computations and graphs were obtained in \cite{MPSS1986}.}
\begin{figure}[ht]
\includegraphics[width=0.45\textwidth]{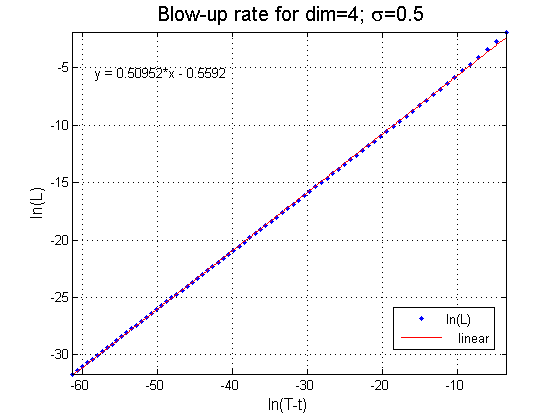}
\includegraphics[width=0.45\textwidth]{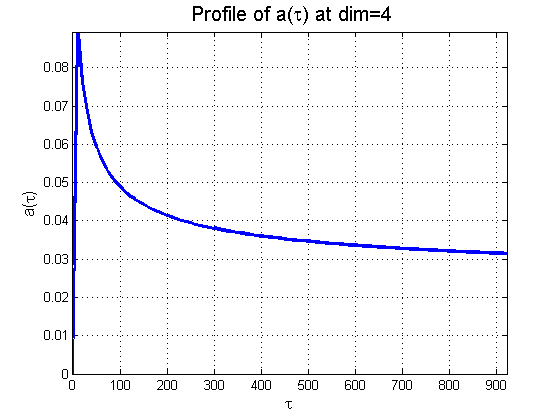}
\caption{The 4d case: the slope of $L(t)$ on the log scale (left); the behavior of $a(\tau)$ (right).}\label{fig:1.1}
\end{figure}

Our next task is to obtain the correction term in the rate \eqref{R:half}. For that, as it is discussed in \cite{SS1999}, one has to study the behavior of $a(\tau)$, defined in \eqref{DRNLS}-\eqref{E:a}. We show the behavior of $a(\tau)$ in Figures \ref{fig:1.1} on the right; the graph shows a very slow decay of $a(\tau)$, which is also similar to the 2d case in \cite{MPSS1986}.

The formal asymptotic analysis in \cite{LPSS1988} (see also \cite[Section 8]{SS1999}) investigates the decay rate of $a(\tau)$. For that, first, substituting $v=e^{i \tau-i a\xi^2/4}w$ into \eqref{DRNLS}, the equation for $w$ is obtained: 
\begin{equation}\label{E:w}
iw_\tau + \Delta w - w +\frac14 b(\tau) \xi^2 w+|w|^{2 \sigma} w = 0,
\end{equation}
where $b(\tau) = a^2 + a_\tau = -L^3 L_{tt}$. Then, studying the asymptotic behavior of solutions to \eqref{E:w}, a description of the asymptotic form of collapsing solutions very close to singularity is given in \cite[Section 8.2]{SS1999}. 
In particular, \cite[Proposition 8.5]{SS1999} shows the following behavior of parameters:
\begin{equation}\label{E:parameters1}
a = -L_t L, \qquad \tau_t = L^{-2}, \qquad L^3L_{tt} = -b,
\end{equation}
where
\begin{equation}\label{E:b}
b = a^2+a_\tau \approx a^2 \quad \mbox{obeys} \quad b_\tau = -\frac{2N_c}{M}\nu(\sqrt{b}) \approx -\frac{2 \nu_0^2}{M} e^{-\pi/\sqrt{b}}
\end{equation}
(thus, $a_{\tau} \ll a^2$)
with
$$
N_c = \int_0^\infty Q^2 r^{d-1} \, dr, \quad M = \frac14 \int r^2 Q^2 r^{d-1} \, dr \quad \mbox{and} \quad
\nu_0 = \lim\limits_{r \to \infty} e^r r^{\frac{d-1}2} Q(r).
$$
The above characterization of parameters was first derived by Fraiman in 1985 \cite{F85} via perturbation arguments for the linear stability analysis around $Q$, and independently by Papanicolaou et al. in \cite{LePSS1988} and \cite{LPSS1988} via a solvability condition from considering dimension as a continuous variable decreasing down to the $L^2$-critical value $\frac2{\sigma}$. At the first leading order, as $\tau \to \infty$, it was shown that $a(\tau) \approx b^{1/2}$ decays at the following rate 
$$
a(\tau) \approx b^{1/2} \sim \frac{\pi}{\ln \tau},
$$ 
i.e., slower than any polynomial rate in $\tau$, see \cite[Proposition 8.6]{SS1999}.
When more corrective terms are retained, then 
$$
\ds a(\tau) \approx \frac{\pi}{\ln \tau + c \cdot \ln \ln \tau},
$$ 
and from $a_{\tau} \approx - \frac{\nu_0^2}{M} a^{-1}e^{-\pi/a}$ (consequence of the solvability condition) one concludes that $c=3$. Furthermore, the scaling factor $L(t)$ has the asymptotic form
\begin{equation}\label{E:loglog-rate}
L(t) \approx \left( \frac{2 \pi (T-t)}{\ln \ln(\frac{1}{T-t})}\right)^{\frac{1}{2}}.
\end{equation}

Following \cite{LPSS1988}, we further investigate numerically the correction term in $L(t)$ in \eqref{R:half} and study the slope of $a(\tau)$ as a function of $1/(\ln \tau + 3 \ln \ln \tau)$. In Figure \ref{fig:1.3}, we  observe that the function $a(\tau)$ vs $1/(\ln \tau +3 \ln \ln \tau)$ is a straight line. It would now be very tempting to make a conclusion that $a(\tau) \sim 1/(\ln \tau+3 \ln \ln \tau)$, as it was also predicted in $d=2$ in \cite{LPSS1988} and \cite{LePSS1988}. However, more care is needed here. 
\begin{figure}[ht]
\includegraphics[width=0.45\textwidth]{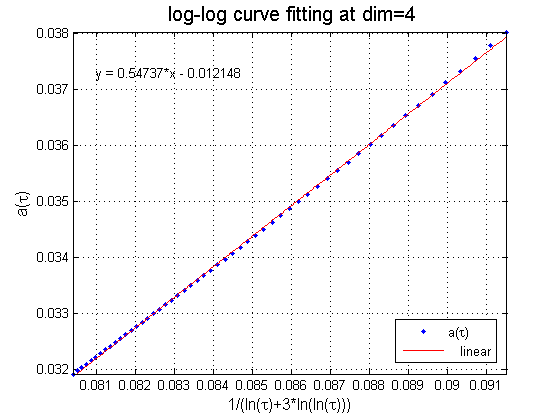}
\caption{\label{fig:1.3} The curve fitting for the $a(\tau)$ for 4d case. This initially suggests that $a(\tau) \sim 1/(\ln \tau + 3\ln \ln \tau)$.}
\end{figure}

If we change the constant $c$ dependence in the second term of $1/(\ln\tau + c \, \ln \ln \tau)$ and try to investigate the slope of $a(\tau)$ as a function of $1/(\ln\tau + c \ln \ln \tau)$ for different $c$, including zero and large constants (we tried $c = 0, 1, 3, 100, 1000$), we find that the slope does not change, which only confirms that such a correction is difficult to investigate numerically, since it only happens at the very high focusing levels. Also, one notes that the slope of the line is not $\frac{1}{\sqrt{2\pi}}$ as expected from asymptotic analysis from \cite{LPSS1988},
where it states that the correction term for $a(\tau)$ is given by 
$$
\ds q(t) \approx \left(\frac{2\pi}{\ln \ln \frac{1}{|T-t|}}\right)^{\frac{1}{2}}.
$$
This is because at the time we stop our simulations (we are forced to stop as the maximal machine's precision is reached), the values of $a(\tau)$ are still far from $0$, this is also observed in \cite{LPSS1988}, \cite{LePSS1988}, \cite{SS1999}. Hence, further justification for the correction term is needed and we discuss it in the next subsection.

\subsubsection{Further justification of ``log-log" correction}
In \cite[Chapter 18]{F2016}, Fibich shows that the log-log regime is reached when the focusing amplitude of the solution is extremely large (e.g., a necessary condition for the log-log to hold in dimension 2 is $A \gg 10^{48}$), which is basically impossible to observe numerically (although see \cite{LDV2013} for even higher order correction terms in the blow-up rate). Instead he suggests studying the reduced equations, where the equations \eqref{E:parameters1}-\eqref{E:b} are written in the following form
\begin{equation}\label{E:reduced-system}
b_{\tau}(\tau)=-\nu(b), \qquad a_{\tau}(\tau)=b-a^2,
\end{equation}
where $\nu(b)=c_{\nu}\, e^{-\pi/\sqrt{b}}$ with $c_\nu = 2 \nu_0^2/M$. 
Our numerical calculations give $c_{\nu} \approx 44.8$ in $d=2$ (matching \S 17.7 in \cite{F2016}) and $c_{\nu} \approx 53.11$ in $d=4$. 
The advantage of working with the system \eqref{E:reduced-system} as mentioned in \cite{F2016} is that it can be solved by standard numerical methods over hundreds of orders of magnitudes without a significant deterioration in the numerical accuracy. 
We solve this system by RK2 and then recover $L(\tau)$ and $T-t$ from numerical quadratures of 
\begin{equation}\label{reduced eqn2}
L(\tau)=L(0)e^{-\int_0^{\tau}a(s)ds}, \qquad T-t=\int_{\tau}^{\infty} \frac{dt}{d\tau}d\tau = \int_{\tau}^{\infty} L^2(\tau)d\tau.
\end{equation}
We show our results for $d=4$ in Figure \ref{reduced 4d}. On the left subfigure in Figure \ref{reduced 4d}, we plot the behavior of $b(t)$ as a function of $L(t)$ for three approximations: a solid blue line is the numerical solution, a dash-dot pink (straight) line is a strict adiabatic approximation $b(t) \equiv b(0) = 10^{-1}$ and the dashed green line is an asymptotic approximation of the log-log law $b(\tau) \sim \frac{\pi^2}{\ln^2\tau}$. One can see that even when the amplitude reaches focusing of $10^{250}$ (i.e., $L(t) \sim 10^{-250}$) the log-log law is further away from the numerical solution than the strict adiabatic law, though it is slowly decreasing down.  
\begin{figure}[ht]
\includegraphics[width=0.45\textwidth]{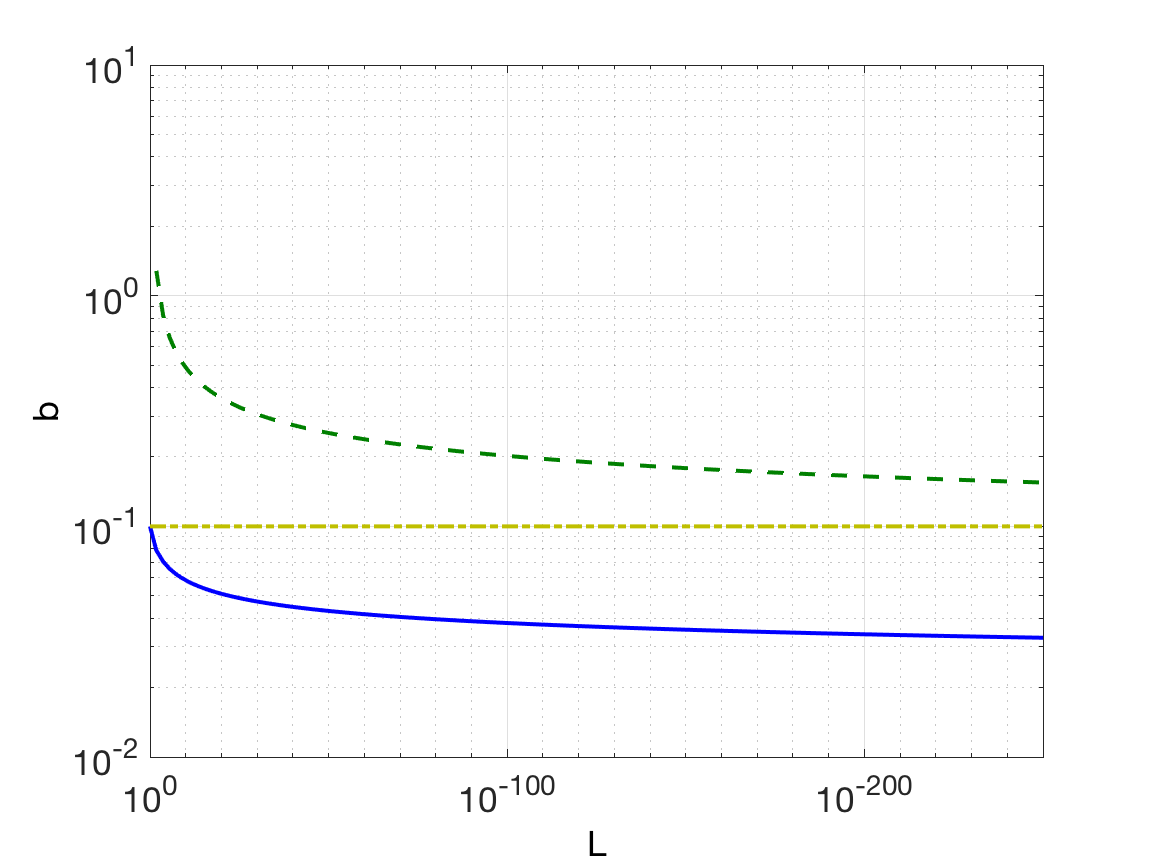}
\includegraphics[width=0.45\textwidth]{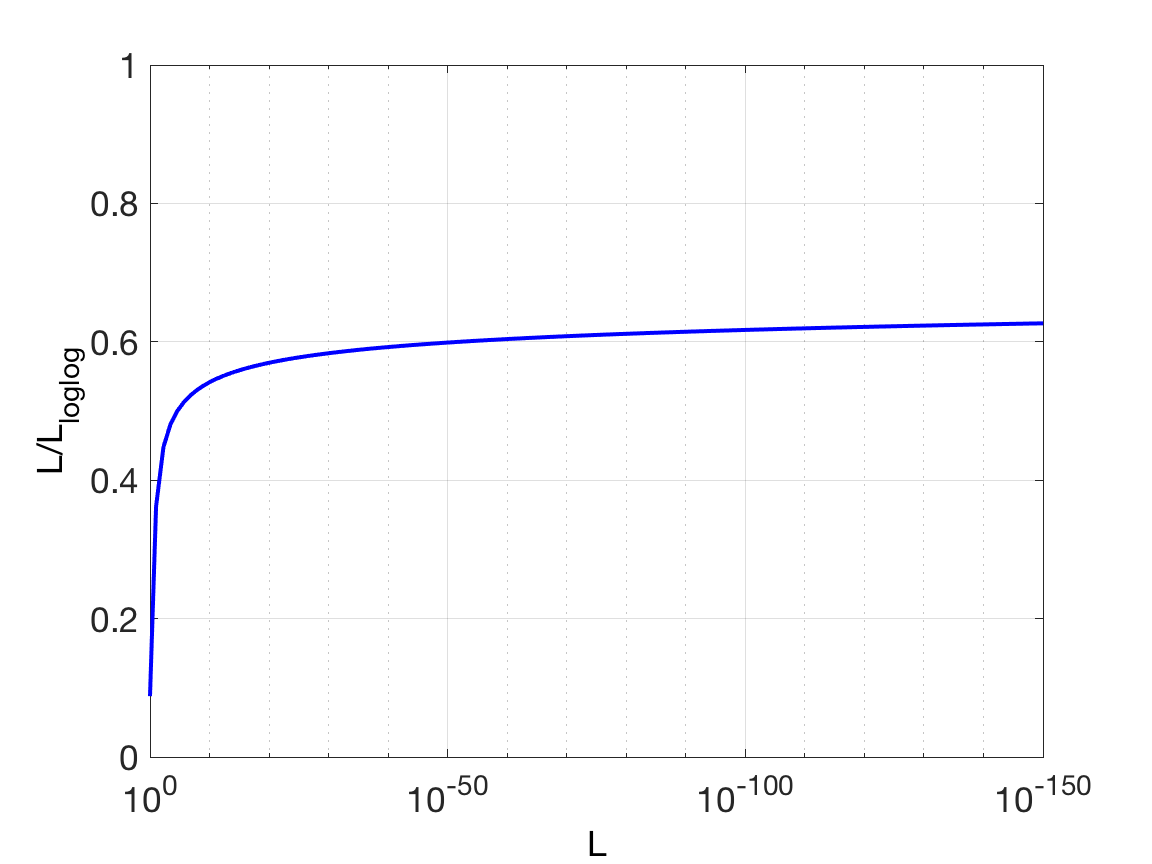}
\caption{The reduced equations \eqref{E:reduced-system}-\eqref{reduced eqn2} in $d=4$, $b(0)=0.1$, $L(0)=1$, and $L_{t}(0)=0$. Left: $b(t)$ as a function of $L(t)$, where solid curve is a numerical solution of the reduced system \eqref{E:reduced-system}, dashdot curve is a strict adiabaticity $b \equiv b(0)$, and dash curve is an asymptotic approximation for the log-log law. Right: the ratio of $L(t)/L_{\log\log}$.}
\label{reduced 4d}
\end{figure}

On the right subfigure in Figure \ref{reduced 4d} the ratio $L(t)/L_{\log\log}$ is shown, where $L_{\log\log}=\left(\dfrac{2\pi(T-t)}{\ln\ln(\frac{1}{T-t})}\right)^{\frac{1}{2}}$ is the anticipated blow-up rate. It is slowly growing towards $1$, however, still a bit far from it. 
These figures also show that there is an intermediate adiabatic regime in the blow-up dynamics, which asymptotically approaches the log-log regime. The adiabatic laws are obtained in \cite{FP1999} and \cite{M1993}, and further discussed in \cite{F2016}.

While it is challenging to observe the log-log correction term numerically, in \cite{ADKM2003} the functional form testing was suggested 
and the authors succeeded in showing that among all tested functional forms, the log-log form stabilized the power $1/2$ in the rate approximation \eqref{FL} and also
minimized the fitting errors the best as the computation increased focusing levels. 
Moreover, the authors were able to show that before the log-log regime, the blow-up mechanism seemed to  follow the adiabatic regime. Specifically, it was supposed in \cite{ADKM2003} that
\begin{equation}\label{FL}
\dfrac{1}{L(t)} \sim \left( \dfrac{F(T-t)}{T-t} \right)^{\frac{1}{2}},
\end{equation}
where $F(s)= \big( \ln \frac{1}{s} \big)^{\gamma}$, $\gamma=1, 0.6, 0.5, 0.4, ...,0$, or $F(s)=\ln \ln \frac{1}{s}$. Then $\frac{1}{L(t_i)}$ was computed at the time $t_i$ as well as the approximation parameter
\begin{equation}\label{rho}
\rho_i=\dfrac{L(t_i)}{L(t_{i+1})}\Big/\ln \left( \dfrac{F_{i+1}/(T-t_{i+1})}{F_{i}/(T-t_{i})} \right).
\end{equation}
Since the power $\rho$ is expected to be $\frac{1}{2}$ (after all it is square root decay), one can check how fast the parameter $\rho_i$ converges to $\frac{1}{2}$ and what choice of $F(s)$ gives the best approximation. In \cite{ADKM2003} it was shown that $F(s)=\ln \ln \frac{1}{s}$ provides the fastest and best parameter $\rho_i$ {\it stabilization}, moreover, it 
also gives the optimal quantity in several error estimates such as the standard deviation, $l^1$-norm discrepancy, and $l^2$-norm discrepancy, which gave an extra assurance of $\ln \ln \frac{1}{s}$ selection for $F(s)$.

In our further numerical investigation of the correction term, we use this functional approach as well. We observe that the ``log-log" correction does stabilize the power $\rho$ close to $\frac12$ the best, and it minimizes the approximation errors quite good (eventually) in the functional fitting. For this purpose, we compute $\frac{1}{L(t)}$ as defined in \eqref{FL}, $\rho_i$ as defined in \eqref{rho}, and $\epsilon_i$, the $l^2$ discrepancy, defined as (following \cite{ADKM2003}) 
\begin{equation}\label{l2-deviation}
\epsilon_i = \left[\frac1{i-j_0+1} \sum_{j=j_0}^{i}\left(\frac{1}{2}-\rho_j\right)^2 \right]^{ \frac{1}{2}}. 
\end{equation}

{
Before proceeding with our fitting results, one other remark is needed. We need to specify the process of calculating the quantity $(T-t_i)$. We take the blow-up time $T$ to be the time when we end our simulation. we have $L(\tau_{i+1})=\exp(\ln L(\tau_{i+1}))$. Denoting $\Delta t_{i+1}:=t_{i+1}-t_i$, and from the last equation of \eqref{rescaled initial},
\begin{align}
\Delta t_{i+1}= \delta \tau L^2(\tau_{i+1}).
\end{align}
Thus, the mapping for rescaled time $\tau$ back to the real time $t$ is calculated as 
\begin{align}
t(\tau_{i+1}) = t((i+1)\delta \tau) := \sum_{j=1}^{i+1} \Delta t_j = \delta \tau\sum_{j=0}^{i+1}  L(\tau_j)^2.
\end{align}

Note that as time evolves, the time difference $T - t(\tau_i)$ will become smaller and smaller, and eventually reach saturation level (with little change), therefore we treat the stopping time $t(\tau_{\text{end}}) = t(\tau_{M})$ as the blow-up time $T$ where $M$ is the total number of iterations when reaching the stopping condition ($L<10^{-17}$). Then, we can take
\begin{align}
T = t(\tau_{\text{end}}) = t(M\delta \tau) = \delta \tau\sum_{j=0}^M  L(\tau_j)^2.
\end{align}
Consequently for any $t_i$, we can calculate $T-t_i$ as 
\begin{align}\label{T-t}
T-t_i=\sum_{j=i+1}^M \Delta t_j = \delta \tau\sum_{j=i+1}^M  L(\tau_j)^2.
\end{align}
This indicates that instead of recording the cumulative time $t_i$, we just need to record the elapsed time between the two recorded data points, i.e., $\Delta t_i =t_{i+1}-t_{i}$. By doing so, it can avoid the loss of significance when adding a small number onto a larger one. }


Continuing with the functional fitting for the correction term, we tried several choices for $F(s)$, namely, $F(s)=1$, $F(s) = \big(\ln \frac1{s}\big)^{\gamma}$, $\gamma = 0.5, 0.4, 0.3, 0.25, 0.2, 0.15, 0.1$ and $F(s)=\ln \ln \frac{1}{s}$.

In the dimension $d=4$, while studying the parameter $\rho_i$, we observe that out of all different powers of $\gamma$ that we tried, mentioned above, in the fitting of $F(s) = \big(\ln \frac1{s} \big)^\gamma$, the powers $0.25$ and $0.2$ seem to approximate the rate $\frac{1}{2}$ the best, see Table \ref{loglog-comparison-4d} for values of $\rho_i$. 
\begin{table}[ht]
\begin{tabular}{|c|c|c|c|c|c|c|c|}
\hline
\multicolumn{8}{|c|}{The fitting power $\rho_i$ from different corrections for $F(s)$}\\
\hline
$i$ &$\frac{1}{L(t)}$ range 
&$1$
&$(\ln \frac{1}{s})^{0.3}$  
&$(\ln \frac{1}{s})^{0.25}$  
&$(\ln \frac{1}{s})^{0.2}$  
&$(\ln \frac{1}{s})^{0.15}$  
&$\ln \ln \frac{1}{s}$ \\
\hline
$0$&$1e4 \sim 3e5 $&$0.5084$&$0.5006$&$0.5019$&$0.5032$&$0.5045$&$0.4997$ \\
\hline
$1$&$3e5 \sim 7e6 $&$0.5056$&$0.5003$&$0.5008$&$0.5018$&$0.5027$&$0.4997$ \\
\hline
$2$&$7e6 \sim 1e8 $&$0.5042$&$0.4998$&$0.5003$&$0.5011$&$0.5019$&$0.4997$ \\
\hline
$3$&$1e8 \sim 2e9 $&$0.5033$&$0.4995$&$0.5000$&$0.5007$&$0.5013$&$0.4997$ \\
\hline
$4$&$2e9 \sim 4e10 $&$0.5028$&$0.4994$&$0.4999$&$0.5004$&$0.5010$&$0.4997$ \\
\hline
$5$&$4e10 \sim 6e11 $&$0.5023$&$0.4994$&$0.4998$&$0.5003$&$0.5008$&$0.4997$ \\
\hline
$6$&$6e11 \sim 8e12 $&$0.5020$&$0.4993$&$0.4997$&$0.5002$&$0.5006$&$0.4997$ \\
\hline
$7$&$8e12 \sim 1e14 $&$0.5018$&$0.4993$&$0.4997$&$0.5001$&$0.5005$&$0.4997$ \\
\hline
$8$&$1e14 \sim 1e15 $&$0.5016$&$0.4993$&$0.4997$&$0.5000$&$0.5004$&$0.4997$ \\
\hline
$9$&$1e15 \sim 2e16 $&$0.5014$&$0.4993$&$0.4996$&$0.5000$&$0.5004$&$0.4997$ \\
\hline
\end{tabular}
\linebreak
\linebreak
\caption{Comparison of the different functional forms $F(s)$ for the correction term in $d=4$. ``$\frac{1}{L(t)}$ range" means the values $\frac{1}{L(t_i)} \sim \frac{1}{L(t_{i+1})}$. The blow-up rate $\rho$ seems to fit close to 0.5 in the intermediate regime for $F(s)=(\ln \frac{1}{s})^{\gamma}$ (for example, for $i=3, 4$ with $\gamma = 0.25$; for $i=7,8,9$ with $\gamma=0.2$), however, the best stabilization of the power is given by the log-log correction.}
\label{loglog-comparison-4d}
\end{table}
In fact, one can first observe that the functional forms $\big(\ln \frac{1}{s}\big)^{\gamma}$ decrease down to the power $\frac{1}{2}$ quite well, for example, $\gamma = 0.25$ gives the best approximation on steps $i=3, 4, 5$ (or corresponding time intervals) as it is the closest to $\frac{1}{2}$; $\gamma = 0.2$ gives the best approximation on steps $i = 6, 7, 8, 9$. However, the second observation is that all such forms tend to decrease down to the power $\frac{1}{2}$ and then continue decreasing further down (for example, the form with $\gamma=0.25$ starts decreasing on steps $i=4, 5,..., 9$ but then underperforms at the step $i=9$ compared to the log-log form), thus, eventually not representing the appropriate correction. Note that the functional form, which stabilizes well and stays quite close to $\frac{1}{2}$, is the log-log form. 

The first observation is due to the existence of an intermediate regime in the blow-up dynamics, the {adiabatic} regime. The second observation is indicating that the adiabatic regime goes asymptotically into the log-log regime.
 
We also computed the $l^2$ discrepancy $\epsilon_i$, defined by \eqref{l2-deviation}, starting from the step $j_0=0$ and accumulating up to the step $i$, and show the results in Table \ref{deviation-comparison-4d}. We note that in the dimension $d=4$ this cumulative error is minimized the best by the log-log correction.  
\begin{table}[ht]
\begin{tabular}{|c|c|c|c|c|c|c|c|}
\hline
\multicolumn{8}{|c|}{The $l^2$ discrepancy $\epsilon$ from different corrections $F(s)$}\\
\hline
$i$ &$\frac{1}{L(t)}$ range 
&$1$
&$(\ln \frac{1}{s})^{0.3}$  
&$(\ln \frac{1}{s})^{0.25}$  
&$(\ln \frac{1}{s})^{0.2}$  
&$(\ln \frac{1}{s})^{0.15}$  
&$\ln \ln \frac{1}{s}$ \\
\hline
$0$&$1e4 \sim 3e5 $&$0.0084$&$6.45e-4$&$0.0019$&$0.0032$&$0.0045$&$3.34e-4$ \\
\hline
$1$&$3e5 \sim 7e6 $&$0.0072$&$4.70e-4$&$0.0015$&$0.0026$&$0.0037$&$3.21e-4$ \\
\hline
$2$&$7e6 \sim 1e8 $&$0.0063$&$4.75e-4$&$0.0012$&$0.0022$&$0.0032$&$3.18e-4$ \\
\hline
$3$&$1e8 \sim 2e9 $&$0.0057$&$5.19e-4$&$0.0011$&$0.0019$&$0.0029$&$3.17e-4$ \\
\hline
$4$&$2e9 \sim 4e10 $&$0.0053$&$5.60e-4$&$9.42e-4$&$0.0017$&$0.0026$&$3.14e-4$ \\
\hline
$5$&$4e10 \sim 6e11 $&$0.0049$&$5.92e-4$&$8.65e-4$&$0.0016$&$0.0024$&$3.11e-4$ \\
\hline
$6$&$6e11 \sim 8e12 $&$0.0046$&$6.15e-4$&$8.07e-4$&$0.0015$&$0.0022$&$3.08e-4$ \\
\hline
$7$&$8e12 \sim 1e14 $&$0.0044$&$6.31e-4$&$7.63e-4$&$0.0014$&$0.0021$&$3.04e-4$ \\
\hline
$8$&$1e14 \sim 1e15 $&$0.0041$&$6.54e-4$&$7.29e-4$&$0.0013$&$0.0020$&$2.99e-4$ \\
\hline
$9$&$1e15 \sim 2e16 $&$0.0040$&$6.50e-4$&$7.00e-4$&$0.0012$&$0.0019$&$2.95e-4$ \\
\hline
\end{tabular}
\linebreak
\linebreak
\caption{Comparison of the $l^2$ discrepancy $\epsilon_i$ in the fitting of different correction terms in $d=4$. One can see that the log-log correction minimizes this error the best.}
\label{deviation-comparison-4d}
\end{table}
Since the log-log regime is an asymptotic regime, we also computed the $l^2$ discrepancy error starting from $j_0 = 7$ and up to the last reliable step $i=9$, i.e., for the window of the three last approximations (see also the discussion about the window of approximations in \cite{ADKM2003}), and recorded it in Table \ref{deviation2-comparison-4d}. 
\begin{table}[ht]
\begin{tabular}{|c|c|c|c|c|c|c|}
\hline
$F(s)$
&$1$
&$(\ln \frac{1}{s})^{0.3}$  
&$(\ln \frac{1}{s})^{0.25}$  
&$(\ln \frac{1}{s})^{0.2}$  
&$(\ln \frac{1}{s})^{0.15}$  
&$\ln \ln \frac{1}{s}$ \\
\hline
$\epsilon$&$0.0017$&$7.28e-4$&$3.24e-4$&$1.06e-4$&$5.00e-4$&$2.69e-4$\\
\hline
\end{tabular}
\linebreak
\linebreak
\caption{Comparison of the $l^2$ discrepancy $\epsilon_i$ computed in the window from $j_0=7$ to $i=9$ in the fitting of different correction terms in $d=4$.} 
\label{deviation2-comparison-4d}
\end{table}
While one can observe that the form $F(s)=\big(\ln \frac{1}{s}\big)^{0.2}$ minimizes the discrepancy $\epsilon_i$ the best in Table \ref{deviation2-comparison-4d}, it is because this specific power of $\gamma = 0.2$ decreases down to $\frac{1}{2}$ the closest at that specific window. However, as discussed above, we expect it to keep decreasing, and thus, getting further away from $\frac{1}{2}$. In general, we suspect that all functional forms $\big(\ln \frac1{s} \big)^{\gamma}$ will for some period of time approximate the power $\frac{1}{2}$ well, but then will escape away from $\frac{1}{2}$, and thus, destabilize away from the blow-up regime. This is, of course, an area for further challenging numerical investigations, as well as possible testing of adiabatic regimes given by Malkin adiabatic law \cite{M1993} and Fibich adiabatic law \cite{F2016}. We note that the second best approximation in Table \ref{deviation2-comparison-4d} is produced by the log-log form. 

For dimension 5 we do a similar investigation and list the results of the functional fittings for $F(s) = 1, \left(\ln \frac1{s} \right)^{0.3}, \left(\ln \frac1{s} \right)^{0.25}, \left(\ln \frac1{s} \right)^{0.2}, \left(\ln \frac1{s} \right)^{0.15}$, and $\ln \ln \frac1{s}$, in Table \ref{loglog-comparison-5d}. 
One can observe that the forms of type $\left(\ln \frac1{s} \right)^{\gamma}$ give decreasing $\rho_i$ as the step $i$ increases; some of them reach the value $0.5$ during the calculated time period ($\gamma = 0.25, 0.2$) and some might reach it eventually ($\gamma = 0.15$). However, the continuing decrease of $\rho_i$ values  does not perform as good as the stabilization seen in the log-log form.  
\begin{table}[ht]
\begin{tabular}{|c|c|c|c|c|c|c|c|}
\hline
\multicolumn{8}{|c|}{The fitting power $\rho_i$ from different corrections for $F(s)$}\\
\hline
$i$ &$\frac{1}{L(t)}$ range 
&$1$
&$(\ln \frac{1}{s})^{0.3}$  
&$(\ln \frac{1}{s})^{0.25}$  
&$(\ln \frac{1}{s})^{0.2}$  
&$(\ln \frac{1}{s})^{0.15}$  
&$\ln \ln \frac{1}{s}$ \\
\hline
$0$&$2e3 \sim 6e4 $&$0.5093$&$0.4998$&$0.5013$&$0.5029$&$0.5045$&$0.4977$ \\
\hline
$1$&$6e4 \sim 2e6 $&$0.5060$&$0.4994$&$0.5005$&$0.5015$&$0.5026$&$0.4989$ \\
\hline
$2$&$2e6 \sim 4e7 $&$0.5044$&$0.4992$&$0.5001$&$0.5009$&$0.5018$&$0.4993$ \\
\hline
$3$&$4e7 \sim 7e8 $&$0.5034$&$0.4991$&$0.4999$&$0.5006$&$0.5013$&$0.4994$ \\
\hline
$4$&$7e8 \sim 1e10 $&$0.5028$&$0.4991$&$0.4997$&$0.5004$&$0.5010$&$0.4995$ \\
\hline
$5$&$1e10 \sim 2e11 $&$0.5024$&$0.4991$&$0.4997$&$0.5002$&$0.5007$&$0.4996$ \\
\hline
$6$&$2e11 \sim 3e12 $&$0.5020$&$0.4992$&$0.4996$&$0.5001$&$0.5006$&$0.4996$ \\
\hline
$7$&$3e12 \sim 5e13 $&$0.5018$&$0.4992$&$0.4996$&$0.5001$&$0.5005$&$0.4996$ \\
\hline
$8$&$5e13 \sim 7e14 $&$0.5016$&$0.4992$&$0.4996$&$0.5000$&$0.5004$&$0.4997$ \\
\hline
$9$&$7e14 \sim 1e16 $&$0.5014$&$0.4992$&$0.4996$&$0.5000$&$0.5003$&$0.4997$ \\
\hline
\end{tabular}
\linebreak
\linebreak
\caption{Comparison of the different functional forms $F(s)$ for the correction terms in $d=5$. ``$\frac{1}{L(t)}$ range" means the values $\frac{1}{L(t_i)} \sim \frac{1}{L(t_{i+1})}$. While all other forms give a slow decrease, the loglog form seems to stabilize to the power $1/2$ the best.}
\label{loglog-comparison-5d}
\end{table}

We supply the $l^2$ discrepancy errors $\epsilon_i$ for the window $j_0 = 7$ to $i=9$ (the last 3 steps) in Table \ref{deviation-comparison-5d}, where the loglog fit has the smallest $l^2$ deviation.   
\begin{table}[ht]
\begin{tabular}{|c|c|c|c|c|c|c|}
\hline
$F(s)$
&$1$
&$(\ln \frac{1}{s})^{0.3}$  
&$(\ln \frac{1}{s})^{0.25}$  
&$(\ln \frac{1}{s})^{0.2}$  
&$(\ln \frac{1}{s})^{0.15}$  
&$\ln \ln \frac{1}{s}$ \\
\hline
$\epsilon$&$0.0017$&$7.98e-4$&$3.81e-4$&$6.74e-4$&$4.65e-4$&$3.41e-4$\\
\hline
\end{tabular}
\linebreak
\linebreak
\caption{The $l^2$ discrepancy $\epsilon$ starting from $j_0=7$ to $i=9$ in the fitting of different correction terms in $d=5$. This means we only consider the behaviors close to blow-up. One can see the log-log  correction minimizes the deviation the best at this stage.}
\label{deviation-comparison-5d}
\end{table}

We provide computations for the functional fittings for other dimensions in Appendix A.

Putting all our numerical calculations, asymptotical analysis and functional fitting results together, we conclude that the blow-up rate $L(t)$ (with the first term correction) is given by \eqref{E:loglog-rate}.

\subsubsection{Blow-up profile}
{
In this subsection we investigate profiles of the blow-up solutions, and show our results in dimensions $d=4$. Figure \ref{p4d} shows how the blow-up solutions $v=v(\xi,\tau)$ from \eqref{DRNLS} converge to the rescaled ground state $Q$, which leads to the conclusion that the profile is given by the rescaled (self-similar) version of the corresponding ground state $\frac1{L^{1/\sigma}} Q (\frac{r}{L})$. }

We plot three different times snapshots ($\tau=2, 40, 400$) and list the time values in both $\tau$ and $t$ variable (it is easier to distinguish and track the profiles in the rescaled $\tau$ variable, as $\tau \to \infty$, than in the variable $t$, which converges to some finite time $0<T<\infty$, and thus, $t$ maybe indistinguishable very close to $T$).
The results for other dimensions are similar, see in \cite{Kai-Thesis}.
\begin{figure}[ht]
\includegraphics[width=0.45\textwidth]{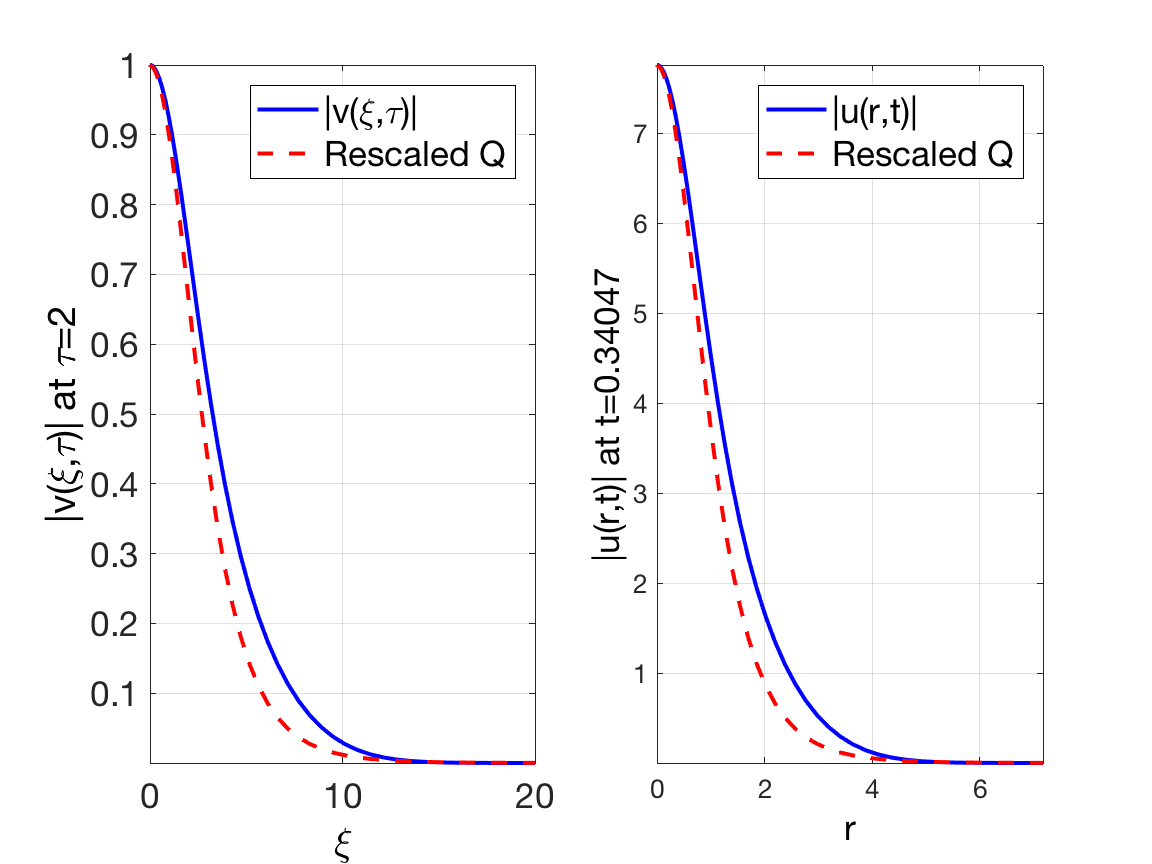}
\includegraphics[width=0.45\textwidth]{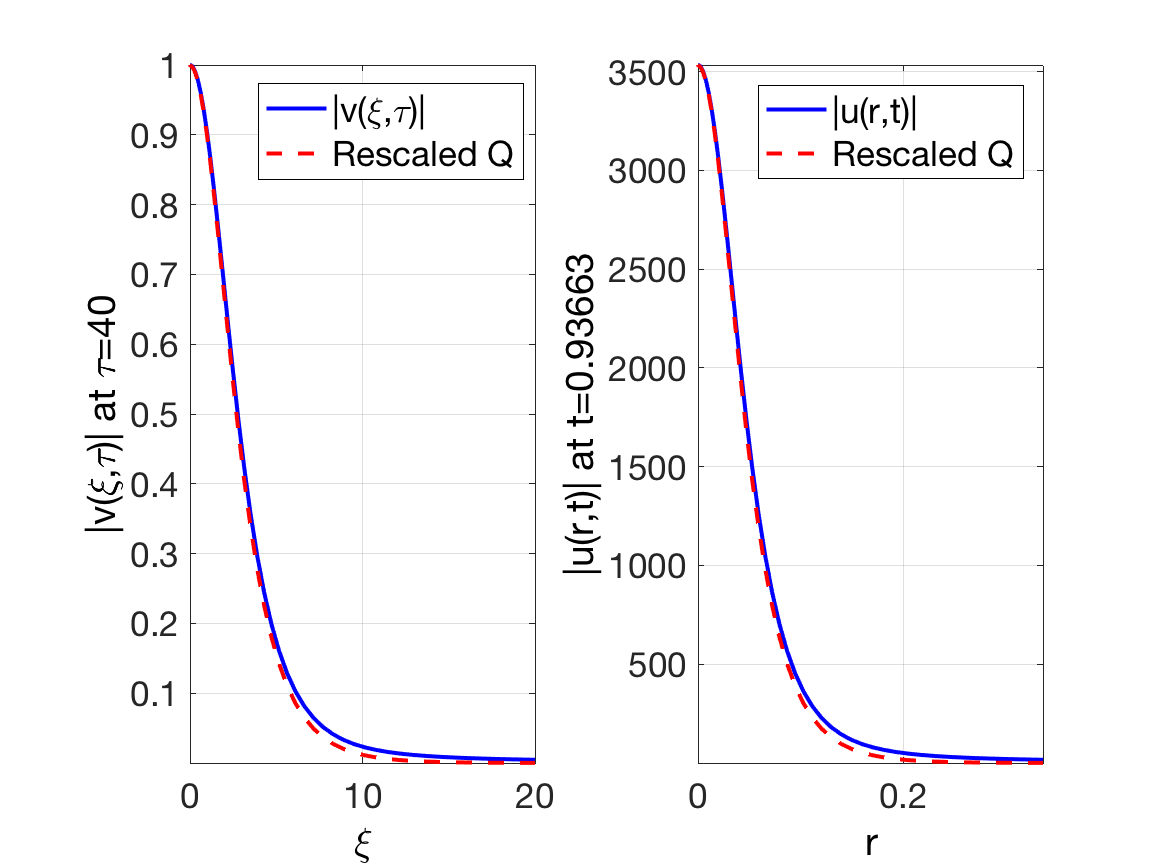}
\includegraphics[width=0.45\textwidth]{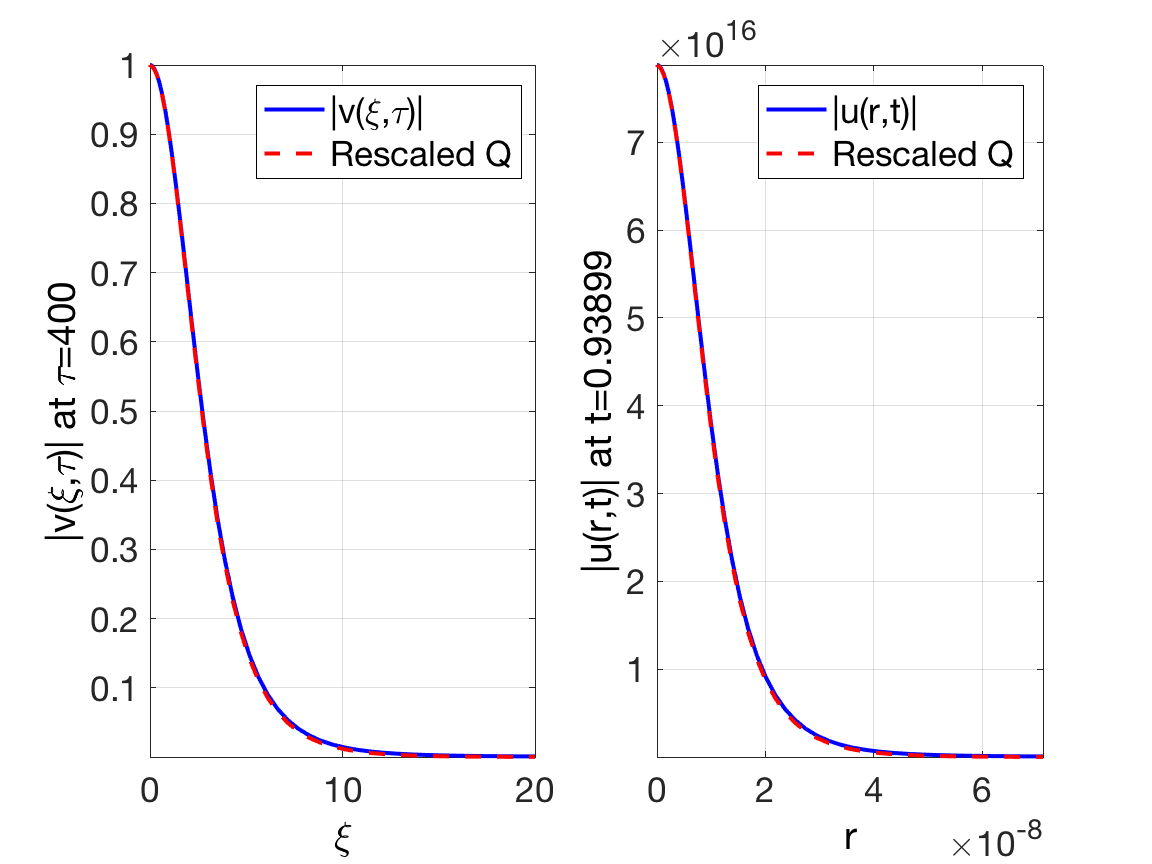}
\caption{Blow-up profile for $d=4$. The blow-up solution converges to the ground state $Q$. Here, the ``Rescaled $Q$" means $\frac{1}{L^{1/\sigma}}Q\left( \frac{r}{L} \right)$.}\label{p4d}
\end{figure}

To summarize, our numerical simulations confirm that a generic blow-up in higher dimensions (in the $L^2$-critical NLS) has the log-log regime characteristics (rate and profile), and our next goal is to justify these observations rigorously.

The analytical proof of the log-log blow-up regime (including higher dimensions) was given in several works of Merle and Rapha\"el, \cite{MR2005, MR2003, MR2004}, provided the Spectral Property 1 holds. In \cite{FMR2006}, the authors were able to check it in dimensions $d =2,3,4$ (up to some corrections in \cite{FMR2008}) and in the next section we investigate higher dimensions $d \geq 5$, while confirming the results for the low dimensions in Appendix E.

\section{Spectral Properties}\label{S:SP}

The stable ``log-log" blow-up regime for the initial data with the mass slightly above the ground state mass, $M[Q]$, negative energy and zero momentum\footnote{or the energy is adjusted for the non-zero momentum} for the 1d case was proved by Merle and Rapha\"{e}l in \cite{MR2005} and \cite{MR2003}, see also the work of Galina Perelman \cite{GP2001} in 1d.
In \cite{MR2005} and \cite{MR2003} a proof of Theorem \ref{T:Main} in higher dimensions was also given, assuming the Spectral Property 1 holds true. A major obstacle for obtaining spectral properties in higher dimensions is the lack of explicit expression for the ground state $Q$. In 2006, Fibich, Merle and Rapha\"el made an attempt to check that the Spectral Property 1 holds true with a numerically-assisted proof for the dimensions up to $d=5$ in \cite{FMR2006}.

Before we proceed to study the Spectral Property 1 in higher dimensions ($d \geq 5)$, we want to check if the methods from \cite{FMR2006} would allow higher dimensional generalization and verification. We first address the choice of the boundary conditions in \cite{FMR2006} (see also \cite{MS2011}). The solution $u(x,t)$ lives in the following space:
$$
\left\lbrace u \,\, \mathrm{is \, radial}  : \,  \int |\nabla u|^2 +|u|^2 e^{-\gamma_0|x|} <\infty \right\rbrace.
$$
This implies that $u_r(r) \rightarrow 0$ as $r\rightarrow \infty$, or equivalently, $u(r) \rightarrow C$ as $r\rightarrow \infty$. On the other hand, from the analysis of the operators $L_1$ and $L_2$, we know that $(1+r^{d-2})u\in L^{\infty}(\mathbb{R}^d)$. In $d=2$, this means that $u\in L^{\infty}(\mathbb{R}^2)$, or in other words, $u(r) \rightarrow C$ as $r\rightarrow \infty$ with $C$ not necessarily being zero. For $d\geq 3$, the last condition implies $|u(r)|<\frac{C}{1+r^{d-2}}$, and hence, $u(\infty)=0$. Thus, using the boundary conditions $u_r(L)=0$ in $d=3,4,5$ in \cite{FMR2006} makes the results less reliable, though the outcome is not affected in $d=3,4$. For clarification, we include Table \ref{comparison} as the comparison between the application of the two different boundary conditions in dimension $d=5$. We also included a comparison for dimensions $d=2,3,4$ in Appendix D.

As far as the higher dimensions $d\geq 6$, we think that one of the reasons that the methods from \cite{FMR2006} can not handle $d\geq 6$ is also the use of the boundary conditions $u_r(L)=0$ for sufficiently large $L$ (say $L=20$ or $L=30$), instead of $u(\infty)=0$. In this paper, we use the boundary condition $u(\infty)=0$. We also note that the same approach was used in \cite{MS2011} for analyzing the 3d cubic NLS equation. Our simulations show that the spectral properties holds for $d \leq 10$ for general case (not necessarily radial), and also for $d=11$ and $d=12$ in the radial case. We stopped our calculations at $d=12$, since the magnitude as well as the $L^2$ norm of the ground state became too large, and, computationally, it was not reliable to guarantee the accuracy. Moreover, the index of both operators $L_1$ and $L_2$ becomes increasingly more challenging to obtain numerically as the dimension $d$ increases beyond $12$.

\subsection{The radial case}\label{S:radial}

In this section, we show that the Spectral Property 1 holds true from $d=5$ to $d=12$, and the Spectral Property 2 holds for $d = 2 ... 12$. 
We first recall the definitions of an index of a bilinear form $B$, see, for example, \cite{FMR2006} and \cite{MS2011}:
\begin{definition}[index of a bilinear form]
The index of a bilinear form $B$ with respect to a vector space $\mathbf{V}$ is
\begin{align*}
\mathrm{ind}_V(B)=\mathrm{min}\lbrace k\in \mathbb{N} \, \vert \,  & \textit{there exists a sub-space P of codimension k such that }\\
&B_{\vert P} \textit{ is positive}.\rbrace
\end{align*}
\end{definition}
Let $B_{1}$ and $B_2$ be the bilinear forms associated with the operators $L_1$ and $L_2$, see (1.10) and (1.11) in Definition (1.1). Note that the $\mathrm{ind}_{H^1}(B_{1,2})$ equals to the number of negative eigenvalues of $L_{1,2}$. Therefore, we often refer to $\mathrm{ind}(L_{1,2})$ as the number of negative eigenvalues of $L_{1,2}$.

Since the potential term $V_{1,2}$ is smooth and decays exponentially fast, according to Theorem XIII.8 in \cite{RS1970}, which is a generalization of the Sturm Oscillation Theorem (Section XIII.7 of \cite{RS1970}), the operators $L_{1,2}$ have finite number of negative eigenvalues. Moreover, the number of the negative eigenvalues can be estimated by counting the number of zeros of the solutions to the following ODE:
\begin{align}\label{index ode1}
\begin{cases}
-\partial_{rr}U-\dfrac{d-1}{r}\partial_rU+V_{1,2}(r)U=0,\\
U(0)=1, \quad U_r(0)=0.
\end{cases}
\end{align}

The ODE \eqref{index ode1} is a standard IVP problem, which can be solved, for example, by matlab solver ``\texttt{ode45}". Note that when $r\gg 1$, the equation \eqref{index ode1} is essentially free (i.e., the potential term can be neglected), see \cite{MS2011}, and consequently, the solution must behave as
\begin{align}\label{index ode free}
U(r)\approx C_1+\dfrac{C_2}{r^{d-2}}.
\end{align}

For the $L_1$ case, we apply the following statement, which is from \cite{FMR2008}. According to this proposition, the numerical values in Table \ref{sign of L1} suggest that there will be no more intersections for $r_0 \geq 6$.

\begin{proposition}[Criterion for the positivity of $u$, \cite{FMR2008}] Let $u$ be a radial solution to
$$-u_{rr}(r)-\frac{d-1}{r}u_r-Vu=0, \quad on \,\, r>0.$$
Let $V_+=max\lbrace V,0 \rbrace$. Assume that there holds for some $r_0>1$
$$\partial_ru(r_0)u(r_0)>0$$
and
\begin{itemize}
\item for $d\geq 3$,
$$ \forall r\geq r_0, \quad V_+(r) \leq \frac{(d-2)^2}{4r^2};$$
\item for $d=2$,
$$ \forall r\geq r_0, \quad V_+(r) \leq \frac{1}{4r^2 (\log r)^2}$$
and
$$ \frac{u'(r_0)}{u(r_0)} \geq \frac{2}{r_0} \int_{r_0}^{\infty}V_+(r)rdr.$$
\end{itemize}
Then $u$ cannot vanish for $r\geq r_0$ (see Table \ref{sign of L1}).
\end{proposition}

\begin{table}[ht]
\begin{tabular}{|c|l|l|l|l|}
\hline
$d$&$5$&$6$&$7$&$8$\\
\hline
$\partial_r u(r_0)u(r_0)$&$1.4009e-05$&$4.7096e-05$&$1.9375e-05$&$6.2832e-06$\\
\hline
$-V_1(r_0)-\frac{d-2}{4r_0}$&$-0.31166$&$-0.53961$&$-0.82591$&$-1.1703$\\
\hline
$d$&$9$&$10$&$11$&$12$\\
\hline
$\partial_r u(r_0)u(r0)$&$1.8316e-06$&$4.802e-07$&$1.0135e-07$&$6.9127e-09$\\
\hline
$-V_1(r_0)-\frac{d-2}{4r_0}$&$-1.5731$&$-2.0343$&$-2.5543$&$-3.1335$\\
\hline
\end{tabular}
\linebreak
\linebreak
\caption{Values of the quantities from Proposition 3.1 at $r_0=6$.}
\label{sign of L1}
\end{table}

For the case of $L_2$, the theorem is not applicable. We adopt the argument from \cite{MS2011}: we first notice that the equation \eqref{index ode1} converges to the free equation \eqref{index ode free}. Then, we choose a large enough interval (say $L=100$) to ensure that the solution goes to a constant. The constant can be found from
\begin{align}\label{index constant}
\begin{cases}
\frac{C_1}{r_i^{d-2}}+C_2=U_i,\\
\frac{C_1}{r_{i-1}^{d-2}}+C_2=U_{i-1},
\end{cases}
\end{align}
where $r_i$ and $U_i$ are discretized points of $r$ and $U(r)$. Once the constant $C_2$ starts to stabilize, we conclude that the solution enters the free region and no more ``zeros" will occur.

We also point out that one needs to be careful in the numerical calculation of the potential $V_1$ or $V_2$, since when $d\geq 5$, the term
$$
V_2=\frac{2}{d}Q^{\frac{4}{d}-1}rQ_r
$$
generates the negative power of $Q$. This may fail to describe the exponential decay property, especially, when $d\geq 8$. An alternative way is needed to calculate the potential $V_1$ and $V_2$. We provide a new approach for that and discuss the details in the Appendix C.

Our numerical solutions of the equation (\ref{index ode1}) are given in Figure \ref{index k0_1} as an example for the case $d=5$, there $U$ stands for the solution to $L_1$ and $Z$ for $L_2$. {Solutions to other dimensions of the equation \eqref{index ode1} are similar.} We conclude the following statement.

\begin{proposition}[indices of $B_{1,2}$]\label{P:indexL}
For $d=5$ to $d=12$, the indices of $L_{1,2}$ in the radial case are
$$\mathrm{ind}(L_1)=2, \quad \mathrm{ind}(L_2)=1$$
Therefore,
$$\mathrm{ind}_{H_r^1}(B_1)=2, \quad \mathrm{ind}_{H_r^1}(B_2)=1$$
\end{proposition}

\begin{figure}
\includegraphics[width=.49\textwidth]{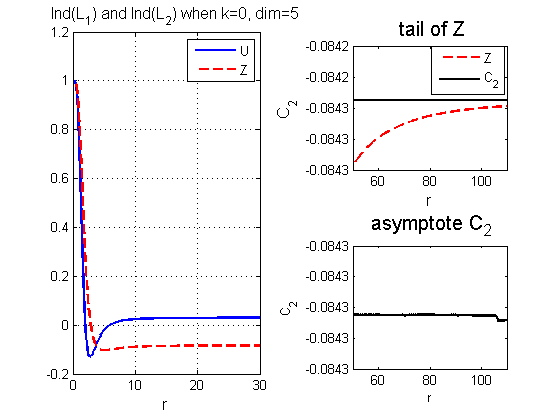}
\caption{\label{index k0_1} Solutions of \eqref{index ode1} in $d=5$. Numerical justification of Proposition \ref{P:indexL}. The blue line shows the behavior of $U$ from $L_1U=0$ and the green line shows the behavior of $Z$ from $L_2Z=0$. We also show the behavior of the tail of $Z$, and the constant $C_2$ it approaches to  as $r\rightarrow \infty$, see \eqref{index constant}.}
\end{figure}

The following property shows that the indices of the bilinear forms are stable under the perturbations. Thus, it is sufficient to check the terms $B_{1,2}(u,u)>0$ instead of $B_{1,2}(u,u)>\delta_0 \int |u|^2e^{-|\mathbf{x}|}$ for some sufficiently small $\delta_0$, given in Definition \ref{D:SP}.

\begin{proposition}[\cite{FMR2006}, \cite{MS2011}]
For the operators $L_{1,2}$ (from the Spectral Property), there exists a universal constant $\delta_0>0$, sufficiently small, such that for the perturbed operators
$$
\bar{L}_{1,2}=L_{1,2}-\delta_0 \, e^{-|x|},
$$
the associated bilinear forms are stable, i.e.,
$$
\mathrm{ind}_{H^1}(\bar{B}_{1,2})=\mathrm{ind}_{H^1}(B_{1,2}).
$$
\end{proposition}

We return to the discussion of the proof of the spectral property, which involves solving the BVP problem $L_{1,2}U=f$. While the numerical calculations suggest that $L_{1,2}$ are invertible, the proof of the invertibility of the operators $L_{1,2}$ in \cite{FMR2006} and \cite{MS2011} works in straightforward adaptation to our cases.

\begin{proposition}[Invertibility of $L_{1,2}$]\label{P:Inv-radial}
Let $d\in \lbrace5,6,\cdots, 12 \rbrace$ and $f\in C^0_{loc}(\mathbb{R}^d)$ be radially symmetric with $|f(r)|\leq e^{-Cr}$. Then there exists a unique radial solution to
$$
L_{1,2}u=f \quad with \, (1+r^{d-2})u \in L^{\infty}.
$$
\end{proposition}

The following definition lists the numerical values of the bilinear forms which we need in the proof of the Spectral Property. It involves the computation of the BVP problem $L_{1,2}U=f$. We take $u_r(0)=0$ as the left boundary condition, since $u(r)$ is radially symmetric. We construct the artificial boundary condition $u(L)+\dfrac{L}{d-2}u_r(L)=0$ to approximate the boundary condition $u(\infty)=0$ (see details in Appendix B, also the reader can refer to \cite{MS2011}).

\begin{definition}[numerical representation of the bilinear form]\label{D:BF}
Let the operator $L_{i}u=f$, $i=1,2$ solve the linear BVP
\begin{align}\label{bvp eqn}
\begin{cases}
L_{1,2}u=f, \quad i=1,2 \\
u_r(0)=0, \quad u(\infty)=0.
\end{cases}
\end{align}
Define
\begin{align}
& L_1U_1=Q, \quad L_1U_2=Q_1;\\
& L_2Z_1=Q_1, \quad L_1Z_2=Q_2;
\end{align}
and denote the following constants as values of the bilinear forms
\begin{align}
& K_{11}=B_1(U_1,U_1), \quad K_{22}=B_1(U_2,U_2), \quad K_{12}=B_1(U_1,U_2),\quad K_{21}=B_1(U_2,U_1); \label{B1-K}\\
& J_{11}=B_2(Z_1,Z_1), \quad J_{22}=B_2(Z_2,Z_2), \quad J_{12}=B_2(Z_1,Z_2), \quad J_{21}=B_2(Z_2,Z_1). \label{B2-J}
\end{align}
We also define the determinants of matrices $K$ and $J$ by
\begin{equation}\label{E:KK-JJ}
KK=K_{11}K_{22}-K_{12}K_{21} \quad \mbox{and} \quad JJ =J_{11}J_{22}-J_{12}J_{21}.
\end{equation}
\end{definition}
\bigskip

We list the values of $K_{ij}$ and $J_{ij}$ from Definition \ref{D:BF} for dimensions $5 \leq d \leq 12$ in Table \ref{K0} and Table \ref{J0}, respectively.

\begin{table}[ht]
\begin{tabular}{|c|l|l|l|l|l|}
\hline
$d$&$K_{11}$&$K_{12}$&$K_{22}$&$KK$&$|\frac{K_{12}}{K_{21}}-1|$\\
\hline
$5$&$-42.3114$&$4.515$&$-490.2964$&$20724.7444$&$5e-10$\\
\hline
$6$&$-279.0336$&$23.4744$&$-2400.2078$&$669187.6405$&$4e-12$\\
\hline
$7$&$-2199.9811$&$294.6847$&$-15415.6211$&$33827236.411$&$8e-11$\\
\hline
$8$&$-20133.2095$&$4529.9623$&$-119781.6837$&$2391069177$&$4e-11$\\
\hline
$9$&$-209271.9449$&$73555.7296$&$-1073544.6307$&$219252327433$&$5e-11$\\
\hline
$10$&$-2428264.8856$&$1256378.6856$&$-10704624.457$&$24415176279909$&$1e-9$\\
\hline
$11$&$-30987268.4148$&$22770115.4386$&$-114540537.6$&$3e16$&$5e-10$\\
\hline
$12$&$-428716358.9148$&$440685490.5273$&$-1248715577$&$3e18$&$1e-7$\\
\hline
\end{tabular}
\linebreak
\linebreak
\caption{Values of the bilinear form $B_1$ from \eqref{B1-K} and \eqref{E:KK-JJ} via $K_i$'s.} 
\label{K0}
\end{table}

\begin{table}[ht]
\begin{tabular}{|c|l|l|l|l|l|}
\hline
$d$&$J_{11}$&$J_{12}$&$J_{22}$&$JJ$&$|\frac{J_{12}}{J_{21}}-1|$\\
\hline
$5$&$80.6653$&$-483.0279$&$1388.5127$&$-121311.0934$&$1e-10$\\
\hline
$6$&$611.2497$&$-3969.9773$&$12386.6137$&$-8189405.1473$&$2e-10$\\
\hline
$7$&$5608.4168$&$-39096.1569$&$134642.4625$&$-773378434.5754$&$1e-9$\\
\hline
$8$&$60626.5104$&$-449322.1199$&$1721581.0897$&$-97516913089.4223$&$1e-9$\\
\hline
$9$&$758310.8674$&$-5924273.2959$&$25329030.7313$&$-15889734594779.6$&$1e-9$\\
\hline
$10$&$10852351.2386$&$-88679869.0076$&$423309818.7944$&$-3270212385731735$&$1e-8$\\
\hline
$11$&$176804567.641$&$-1500253955.9189$&$7990965776.4667$&$-8e17$&$1e-7$\\
\hline
$12$&$3286852523.1216$&$-28762231494.6882$&$171112628940.871$&$-3e20$&$1e-6$\\
\hline
\end{tabular}
\linebreak
\linebreak
\caption{Evaluation of the bilinear form $B_2$ from \eqref{B2-J} and \eqref{E:KK-JJ} via $J_i$'s.} 
\label{J0}
\end{table}

We use two methods to solve the equation (\ref{bvp eqn}): one is the Chebyshev collocation method; the other method is the matlab solver ``\texttt{bvp4c}". These two methods lead to basically the same results, for a comparison in dimension $d=5$ see Table \ref{compare} (values of $K_i$'s) and Table \ref{comparison} (values of $J_i$'s). Since $L_1$ and $L_2$ are self-adjoint operators, the difference $|\frac{K_{12}}{K_{21}}-1|$, and the corresponding one for $J$'s, is one way to check the numerical consistency, we list those values in the last columns of the Tables \ref{K0}-\ref{J0}; in Tables \ref{compare}-\ref{comparison} we list the differences $|K_{12} -K_{21}|$ and $|J_{12} -J_{21}|$, correspondingly. 
\begin{table}[h]
\begin{tabular}{|c|l|l|l|l|l|}
\hline
$d=5$&$K_{11}$&$K_{12}$&$K_{22}$&$KK$&$|K_{12}-K_{21}|$\\
\hline
``cheby"&$-42.3114$&$4.515$&$-490.2964$&$20724.7444$&$5e-10$\\
\hline
``\texttt{bvp4c}"&$-42.3114$&$4.515$&$-490.2966$&$20724.7515$&$3e-05$\\
\hline
\end{tabular}
\linebreak
\linebreak
\caption{Comparison of values of $K_i$'s in $d=5$ between Chebyshev-collocation method and ``\texttt{bvp4c}". The boundary condition used here in both methods is $u(\infty)=0$. }
\label{compare}
\end{table}

We note that we use $u(\infty) = 0$ boundary condition when computing the values of $K$'s in Table \ref{K0} as well as in Table \ref{compare}; moreover, in Table \ref{compare} we provide results obtained by two methods for comparison purposes.

In the Table \ref{comparison} we show two different boundary conditions: $u(\infty) = 0$ (first row with Chebyshev collocation method)
and $u'(L)=0$ with Chebyshev collocation methods (second row) and with ``\texttt{bvp4c}" from matlab (third row). This is for comparison of the results with lower dimensions, since it makes a difference in dimension 4 (though it does not influence the signs, thus, the conclusion of the Spectral Property), and it completely changes the results in dimension 5 (and higher).
Therefore, starting from the dimension 5 and higher, we only use boundary condition as in the first row of Table \ref{compare}.
\begin{table}[h]
\begin{tabular}{|c|c|c|c|c|c|}
\hline
B.C.&$J_{11}$&$J_{22}$&$J_{12}$&$JJ$&$|J_{12}-J_{21}|$\\
\hline
$u(\infty)=0$ by ``cheby"&$80.6653$&$1388.5127$&$-483.0279$&$-121311.0935$&$7e-8$\\
\hline
$u'(L)=0$ by ``cheby"&$-114.8176$&$-4100.538$&$552.837$&$165185.398$&$2e-4$\\
\hline
$u'(L)=0$ by ``\texttt{bvp4c}"&$-114.8176$&$-4100.5315$&$552.8372$&$165185.2598$&$2e-3$\\
\hline
\end{tabular}
\linebreak
\linebreak
\caption{Comparison of values of $J_i$'s for $d=5$. Note that the boundary conditions affect the final results.}
\label{comparison}
\end{table}

With these bilinear forms calculated, we reach the following proposition:
\begin{proposition}\label{P:B-space}
The bilinear form $B_1(f,f)$ is coercive on the space $\mathcal{U}=\lbrace Q,Q_1\rbrace^{\perp}$ and $\ \mathcal{U} \subseteq H^1_r$, where $H^1_r$ stands for the radial functions in $H^1$.
The bilinear form $B_2(g,g)$ is coercive on the space $\mathcal{V}=\lbrace Q_1,Q_2\rbrace^{\perp}$ and $\mathcal{V} \subseteq H^1_r$.
Therefore, the Spectral Property 1 holds in the functional space $\mathcal{U} \times \mathcal{V} \subseteq H^1_r \times H^1_r$.
\end{proposition}
\begin{proof}

We outline the key steps of the proof, we refer to \cite{FMR2006}, \cite{MS2011} or \cite{SZ2011} for the details, as they are the same.
Let's consider the form $B_1$ and recall from Proposition 3.2 that the $\mathrm{ind}_{H^1_r}(B_1)=2$. From the Table \ref{K0} we have $K_{11}=B_1(U_1,U_1)<0.$ This suggests that $U_1$ is one of the negative spans of $L_1$ associated with $B_1$. Similarly, we have $U_2$ is the other negative span of $L_1$ associated to $B_1$, since $K_{22}<0$. Moreover, the determinant $KK>0$ suggests that the matrix from their linear combinations is negative definite, and thus, the decomposition is non-degenerate. In summary, since we exhibit two negative spans and $\mathrm{ind}(B_1)=2$, the remaining spans, that are orthogonal to $(U_1, U_2)$ in the sense of $B_1$, or equivalently, orthogonal to $Q$ or $Q_1$ in $L^2$ sense, must generate the positive outcome, i.e., $B_1(f,f)>0$ if $\langle f,Q \rangle=\langle f,Q_1 \rangle=0.$

Next, we consider the bilinear form $B_2$, while $J_{11}=B_2(Z_1,Z_1)$ and $J_{22}=B_2(Z_2,Z_2)$ generate the positive values, we check their linear combination $\bar{Z}=Z_1+\alpha Z_2$ with $\alpha=-\frac{J_{12}}{J_{22}}$. This value comes from the fact that if we calculate the quadratic form, this constant gives the minimum value for the bilinear form. Moreover, $B_2(\bar{Z},\bar{Z})=C\cdot JJ$ for some constant $C>0$, i.e., the value of the bilinear form has the same sign as the determinant. Thus, $JJ<0$ suggests $B_2(\bar{Z},\bar{Z})<0$ and $\bar{Z}$ lies on the negative span of $L_2$ as we desire. Therefore, any $g$, which is orthogonal to the direction $\bar{Z}$, i.e., $B_2(g,\bar{Z})=0$, or equivalently, $\langle g,Q_1 \rangle=\langle g,Q_2 \rangle=0$, is orthogonal to $(Z_1, Z_2)$, since $\bar{Z}$ is their linear combination, and consequently, we have $B_2(g,g)>0$, which justifies the statement of Proposition \ref{P:B-space}. 

To make the argument rigorous, note that the functions $f$ and $g$, considered above, are not in $H^1_r$ (from the Proposition \ref{P:Inv-radial}, we have $(1+r^{d-2})f \in L^{\infty}(\mathbb{R}^d)$ and a similar requirement for $g$). We introduce an appropriate cut-off function and then take the limit. For further details, see \cite{FMR2006}, \cite{MS2011} or \cite{SZ2011}.
\end{proof}

Finally, we provide details on the Spectral Property 2, where the orthogonal conditions for the second bilinear form can be changed to just one condition. 
We set $L_2 Z =Q$ and compute the quantity $\langle L_2 Z,Z \rangle$.
Table \ref{T:L2Q} contains the values for the quantity $\langle L_2 Z, Z \rangle$ in dimensions 2 to 6 with two different boundary conditions, and then in dimensions 7 to 12 with $u(\infty)=0$.
\begin{table}[H]
\begin{tabular}{|c|c|c|c|c|c|c|}
\hline
Boundary conditions&$d$&$2$&$3$&$4$&$5$&$6$\\
\hline
$u'(L)=0$&$\langle L_2 Z,Z \rangle$&$-2.9513$&$-4.8101$&$39.6029$&$-173.4202$&$-14058.3264$\\
\hline
$u(\infty)=0$&$\langle L_2 Z,Z \rangle$&NA&$-6.7563$&$-23.3181$&$-107.8715$&$-623.5943$\\
\hline
$d$&$7$&$8$&$9$&$10$&$11$&$12$\\
\hline
$\langle L_2 Z,Z \rangle, u(\infty)=0$&$-4290.4717$&$-33763.4$&$-291711$&$-2604677$&$-20638152$&$-33716967$\\
\hline
\end{tabular}
\caption{Value of $\langle L_2 Z,Z \rangle$ for different dimensions, where $L_2 Z=Q$.}
\label{T:L2Q}
\end{table}
\begin{remark}\label{R:1}
Since by Proposition \ref{P:indexL} the index of $L_2$ is one, arguing as in Proposition \ref{P:B-space} and incorporating the results from Table \ref{T:L2Q}, we obtain that  the spectral property holds for $L_2$ only with one condition: $\langle Q,g \rangle = 0$, which proves Theorem \ref{T:SP2},  and thus,  validates the Spectral Property 2 in the radial case. For completeness, we include details in the general case in Theorem \ref{T:nonradial2}.
\end{remark}

\subsection{The non-radial case} 

To study the non-radial case, we rewrite our operators in the form of spherical harmonics, i.e.,
\begin{align}\label{L12 sh}
L_{1,2}^{(k)}=-\partial_{rr}- \frac{d-1}{r} \partial_r +V_{1,2}(r)+\frac{k(k+d-2)}{r^2}, \quad k=0,1,2,... .
\end{align}
The notation $L_{1,2}^{(k)}$ and $B_{1,2}^{(k)}$ will stand for the $k$th spherical harmonics. When $k=0$, it is simply the radial case, which we already discussed in Section \ref{S:radial}. For $k > 0$ the bilinear forms $B_{1,2}^{(k)}$ have the same properties as $L_{1,2}^{(k)}$, similar to the case $k=0$ (see \cite{RS1970}, \cite{FMR2006} and \cite{MS2011}). We first study indices of the $k$th bilinear forms. 

\begin{proposition}[indices in the non-radial case]\label{P:ind-nonradial}
The index of the bilinear form $\mathrm{ind}(B_{1,2}^{(k)})$ equals the number of zeros of the IVP problem, counting from $r>0$:
\begin{align}\label{index eqn sh}
\begin{cases}
&L_{1,2}^{(k)} U^{(k)}=0,\\
&\ds\lim_{r\rightarrow 0}\dfrac{U^{(k)}(r)}{r^k}=1,\\
&\ds\lim_{r\rightarrow 0} \dfrac{d}{dr} \dfrac{U^{(k)}(r)}{r^k}=0.
\end{cases}
\end{align}
Numerical calculations show that for $d=4,5,6,7,8,9,10$, we have
$$\mathrm{ind}(L_1^{(1)})=1, \quad \mathrm{ind}(L_1^{(2)})=0, \quad \mathrm{ind}(L_2^{(1)})=0. $$
Consequently,
$$\mathrm{ind}_{H^1}(B_1^{(1)})=1, \quad \mathrm{ind}_{H^1}(B_1^{(2)})=0, \quad \mathrm{ind}_{H^1}(B_2^{(1)})=0.$$
For $d=11,12$, we get
$$ \mathrm{ind}(L_1^{(1)})=1, \quad \mathrm{ind}(L_1^{(2)})=1, \quad \mathrm{ind}(L_1^{(3)})=0, $$
$$ \mathrm{ind}(L_2^{(1)})=1, \quad \mathrm{ind}(L_2^{(2)})=0, $$
and consequently,
$$\mathrm{ind}_{H^1}(B_1^{(1)})=1, \quad \mathrm{ind}_{H^1}(B_1^{(2)})=1, \quad ind_{H^1}(B_1^{(3)})=0.$$
$$\mathrm{ind}_{H^1}(B_2^{(1)})=1, \quad \mathrm{ind}_{H^1}(B_2^{(2)})=0. $$
\end{proposition}
This proposition is obtained from the numerical solutions of the IVP in \eqref{index eqn sh} in the corresponding cases, see Figures \ref{index 5d k1 k2} as an example in the case $d=5$. {Figures \ref{d11 index k2} and \ref{d11 index k3} show the solution of \eqref{index eqn sh} in cases $d=11,12$ and $k=2,3$, since they are different from the cases $d=5,\cdots, 10$.} Here, $U$ stands for the solution of $L_1^{(k)}$, and $Z$ for $L_2^{(k)}$.

In this non-radial case, in our numerical simulations, we want to get rid of the limit terms in the boundary conditions. We use the approach from \cite{MS2011}: let $U^{(k)}(r)=r^k\tilde{U}^{(k)}(r)$, then the operator $L_i^{(k)}$ becomes
\begin{align}
\tilde{L}_{1,2}^{(k)}=\partial_{rr}-\frac{d-1+2k}{r}\partial_r+V_{1,2}.
\end{align}
We rewrite (\ref{index eqn sh}) as follows
\begin{align}\label{index eqn}
\begin{cases}
r^k \left(-\partial_{rr} \tilde{U}- \dfrac{d-1+2k}{r} \partial_r  \tilde{U}+V_{1,2}(r)\tilde{U}\right)=0\\
\tilde{U}(0)=1,\quad \tilde{U}_r(0)=0.
\end{cases}
\end{align}

\begin{figure}
\includegraphics[width=0.45\textwidth]{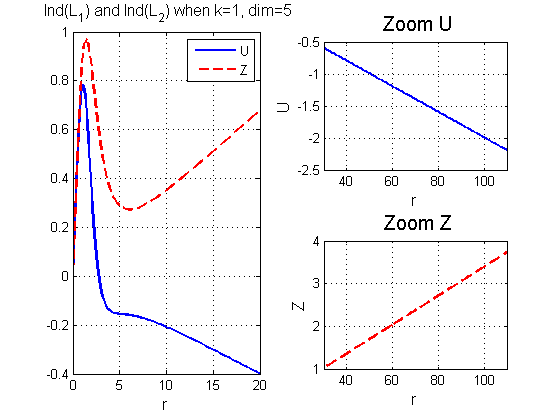}
\includegraphics[width=0.45\textwidth]{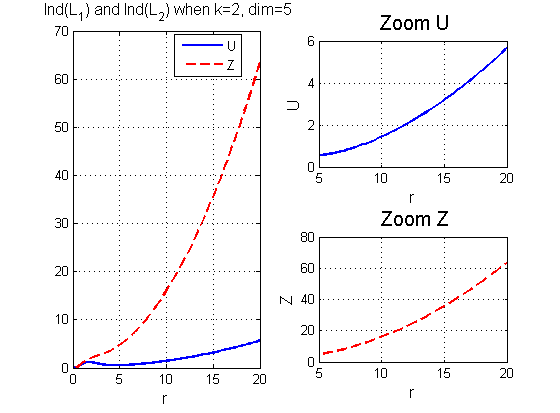}
\caption{\label{index 5d k1 k2} Solutions of \eqref{index eqn sh} in $d=5$. Numerical justification of the Proposition \ref{P:ind-nonradial}. Left figure is the for $k=1$ and right figure is for $k=2$. The blue line shows the behaviors of $U$ from $L_1^{(k)}U=0$ and the red line shows the behaviors of $Z$ from $L_2^{(k)}Z=0$. To the right of each plot are zooms of the tails for $U$ and $Z$. One can see the tails of $U$ and $Z$ increase or decrease with a rate $r^k$, which justifies the asymptotic behavior in \eqref{U: asymptote}.}
\end{figure}

\begin{figure}
\includegraphics[width=0.45\textwidth]{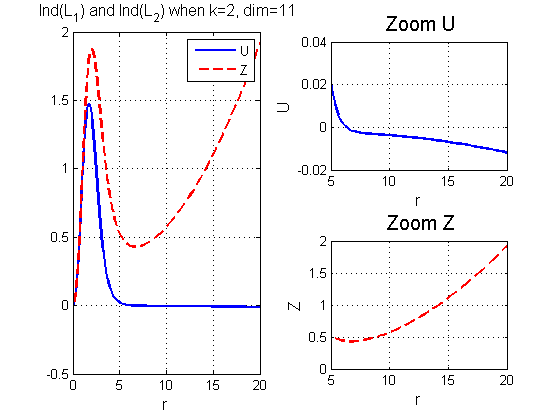}
\includegraphics[width=0.45\textwidth]{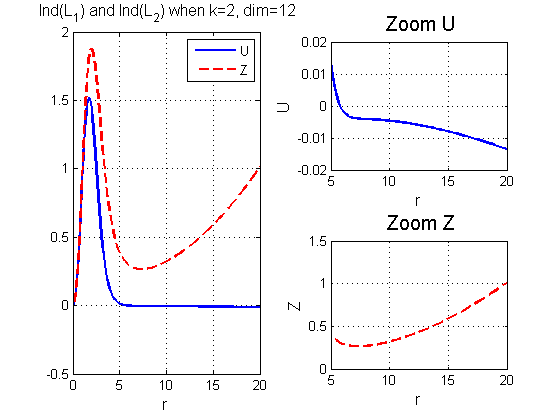}
\caption{\label{d11 index k2} Solutions of \eqref{index eqn sh} in $d=11$ (left) and $d=12$ (right), or numerical justification of the Proposition \ref{P:ind-nonradial}. The blue line shows the behaviors of $U$ from $L_1^{(2)}U=0$ and the red line shows the behaviors of $Z$ from $L_2^{(2)}Z=0$. To the right of each plot are zooms of the tails for $U$ and $Z$. One can see $U$ and $Z$ increase or decrease with a quadratic rate. This justifies the asymptotic behavior in \eqref{U: asymptote} for $k=2$.}
\end{figure}

\begin{figure}
\begin{center}
\includegraphics[width=0.45\textwidth]{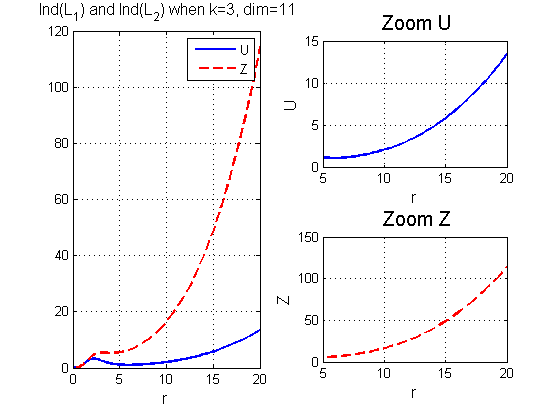}
\includegraphics[width=0.45\textwidth]{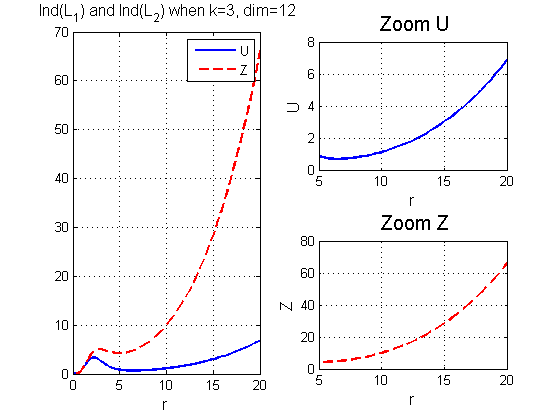}
\caption{\label{d11 index k3} Solutions of \eqref{index eqn sh} in $d=11$ (left) and $d=12$ (right), or numerical justification of the Proposition \ref{P:ind-nonradial}. The blue line shows the behaviors of $U$ from $L_1^{(3)}U=0$ and the red line shows the behaviors of $Z$ from $L_2^{(3)}Z=0$. To the right of each plot are zooms of the tails for $U$ and $Z$. One can see $U$ and $Z$ increase or decrease with a cubic rate. This justifies the asymptotic behavior in \eqref{U: asymptote} for $k=3$.}
\end{center}
\end{figure}

The IVP (\ref{index eqn}) can be solved by matlab solver ``\texttt{ode45}". Then the solution $U$ can be reconstructed by $U(r)=r^k\tilde{U}(r)$. From \cite[Sec. 3]{MS2011}, it follows that $U(r)$ satisfies the asymptotic behavior for $r\gg 1$
\begin{align}\label{U: asymptote}
U\approx C_1r^k+C_2r^{2-d-k}.
\end{align}
This indicates that for $r \gg 1$, $U(r)$ either grows or decays with a polynomial rate $r^k$. Consequently, no more zeros will occur. We stop calculations once we see such  polynomial increase or decrease.

The following property (\cite{RS1970}, also see \cite{FMR2006}, \cite{MS2011}) shows that once we found some $k_0$ such that $\mathrm{ind}(L_{1,2}^{(k_0)})=0$, we can stop the calculation. This avoids checking infinitely many $k$'s.
\begin{proposition}
The index is monotonic with respect to $k$, that is,
$$
\mathrm{ind}(B_{1,2}^{(k+1)})\leq \mathrm{ind}(B_{1,2}^{(k)}).
$$
\end{proposition}

Moreover, the uniqueness and stability of the indices of the bilinear forms $B_{1,2}$ can also be extended to the non-radial case:

\begin{proposition}[Stability, \cite{FMR2006}, \cite{MS2011}]
For the operators $L^{(k)}_{1,2}$ defined in (\ref{L12 sh}), there exists a universal constant $\delta_0>0$, sufficiently small, such that for the perturbed operators
$$\bar{L}_{1,2}^{(k)}=L_{1,2}^{(k)}-\delta_0e^{-|{x}|},$$
the associated bilinear forms are stable:
$$\mathrm{ind}_{H^1}(\bar{B}^{(k)}_{1,2})=\mathrm{ind}_{H^1}(B^{(k)}_{1,2}).$$
\end{proposition}

\begin{proposition}[Invertibility, \cite{FMR2006}, \cite{MS2011}]
Let $d\in [5,12]$ and $f\in C^0_{loc}(\mathbb{R}^d)$ with radial symmetry and $|f(r)|\leq e^{-Cr}$. Then there exists a unique radial solution to
$$L_{1,2}^{(k)}u=f \quad with \, (1+r^{d-2+k})u \in L^{\infty}(\mathbb{R}^d).$$
\end{proposition}

Similar to the radial case, we list the numerical values of the bilinear forms, needed to prove the spectral properties for the non-radial case.

\begin{definition}[numerical computations for the bilinear form]
Let the operators $L_i^{(k)}$, $i=1,2$, solve the linear BVP
\begin{align}\label{bvp eqn2}
\begin{cases}
L^{(k)}_{1,2}u=f,\\
u_r(0)=0,\quad u(\infty)=0.
\end{cases}
\end{align}
Define
\begin{align}
& L_1^{(1)}U_1^{(1)}=rQ, \quad  L_2^{(1)}Z_1^{(1)}=Q_r;\\
& L_1^{(2)}U_1^{(2)}=rQ.
\end{align}
and denote the following constants as values of the bilinear forms
\begin{align}
& K_{11}^{(1)}=B_1^{(1)}(U_1^{(1)},U_1^{(1)}), \quad J_{11}^{(1)}=B_2^{(1)}(Z^{(1)}_1,Z^{(1)}_1), \quad
K_{11}^{(2)}=B_1^{(2)}(U_1^{(2)},U_1^{(2)}).
\end{align}
\end{definition}
The values of $K_{ij}^{(k)}$'s and $J_{ij}^{(k)}$'s are listed in Tables \ref{K1}, \ref{J1} and \ref{K2}.
\begin{table}[h]
\begin{tabular}{|l|l|l|l|l|}
\hline
$d$&$5$&$6$&$7$&$8$\\
\hline
$K_{11}^{(1)}$&$-338.0072$&$-2514.4344$&$-21858.3993$&$-216907.232$\\
\hline
$d$&$9$&$10$&$11$&$12$\\
\hline
$K_{11}^{(1)}$&$-2409681.326$&$-29422347$&$-386097788.7178$&$-5238134702.1673$\\
\hline
\end{tabular}
\linebreak
\caption{Values of $K^{(1)}_{11}$ corresponding to the form $B^{(1)}_1$.}
\label{K1}
\end{table}

\begin{table}[h]
\begin{tabular}{|l|l|l|l|l|}
\hline
$d$&$5$&$6$&$7$&$8$\\
\hline
$J_{11}^{(1)}$&$689.5337$&$6433.9656$&$73995.9077$&$1055378.5046$\\
\hline
$d$&$9$&$10$&$11$&$12$\\
\hline
$J_{11}^{(1)}$&$19767639.0831$&$648419143$&$-17491047782$&$-69767861282$\\
\hline
\end{tabular}
\linebreak
\caption{Values of $J^{(1)}_{11}$ corresponding to the form $B^{(1)}_2$.}
\label{J1}
\end{table}

\begin{table}[h]
\begin{tabular}{|l|l|l|}
\hline
$d$&$11$&$12$\\
\hline
$K_{11}^{(2)}$&$52218994506811$&$34912990629971056$\\
\hline
\end{tabular}
\linebreak
\caption{Values of $K^{(2)}_{11}$ corresponding to the form $B^{(2)}_1$.}
\label{K2}
\end{table}

The values of these bilinear forms are computed by Chebyshev collocation method with $N=1025$ collocation points on the interval $L=100$. The artificial boundary condition is constructed in a similar way as in the previous Section \ref{S:radial} to approximate $u(\infty)=0$, see also Appendix C for details on artificial boundary condition.

We are now ready to establish both spectral properties. 
\begin{theorem}\label{T:nonradial1}
Let the space $\mathcal{U}=\lbrace Q, Q_1, x_iQ \rbrace ^{\perp} \subseteq H^1$, and the space $\mathcal{V}=\lbrace Q_1, Q_2, Q_{x_i} \rbrace^{\perp} \subseteq H^1$, where $i=1,2,\cdots,d$. In the dimensions $d \leq 10$, the Spectral Property 1 holds on the given subspace $\mathcal{U} \times \mathcal{V} \subseteq H^1\times H^1 $. In the dimension $d=11$ or $d=12$, the Spectral Property 1 holds only on the given subspace $\mathcal{U} \times \mathcal{V} \subseteq H^1_r\times H^1_r $. It is indecisive in the non-radial case.
\end{theorem}

\begin{proof}
We only outline the main idea, as it follows the proof in \cite{FMR2006} and \cite{MS2011}. For $d\leq 10$, $\mathrm{ind}(L_1^{(1)})=1,$ and we see $K_{11}<0$ from Table \ref{K1}. Thus, $U_1^{(1)}$ generates the negative span. For any $f \in \lbrace U_1^{(1)} \rbrace^{\perp}$ in the $B_1^{(1)}$ sense, $B_1^{(1)}(f,f)>0$. For $L_1^{(2)}$ and $L_2^{(1)}$, $\mathrm{ind}(L_1^{(2)})=\mathrm{ind}(L_2^{(1)})=0$, thus $B_1^{(2)}$ and $B_2^{(1)}$ are coercive. The coercivity of both $B_1$ and $B_2$ implies the spectral property in the given space $\mathcal{U} \times \mathcal{V}$. We need to note that $f$ and $g$ may not necessarily in $H^1_r$, since from the Proposition 3.9, $(1+r^{d-2+k})f \in L^{\infty}(\mathbb{R}^d)$ and the same for $g$. Similar to the radial case we discuss above, we need to introduce an appropriate cut-off function and then take the limit as what the authors did in \cite{FMR2006}, \cite{MS2011} or \cite{SZ2011}.

For the case $d=11$ or $d=12$. $B_1^{(1)}$ and $B_2^{(1)}$ are coercive on the subspace of $H^1_r$ that satisfies the orthogonal conditions from $K_{11}<0$ and $J_{11}<0$. However, when $k=2$, the index of $L_1^{(2)}$ is still one and we get the positive values from the bilinear forms (see Table \ref{K2}). Thus, the coercivity of $B_1^{(2)}$ becomes indecisive. Hence, with current computations, we can only show the spectral property under the radial assumption for $d=11$ or $d=12$.
\end{proof}

\begin{theorem}\label{T:nonradial2}
Let the space $\mathcal{U}=\lbrace Q, Q_1, x_iQ \rbrace ^{\perp} \subseteq H^1$, and the space $\mathcal{V}=\lbrace Q, Q_{x_i} \rbrace^{\perp} \subseteq H^1$, where $i=1,2,\cdots,d$. In the dimensions $d \leq 10$, the Spectral Property 2 holds on the given subspace $\mathcal{U} \times \mathcal{V} \subseteq H^1\times H^1 $. In the dimension $d=11$ or $d=12$, the Spectral Property 2 holds only on the given subspace $\mathcal{U} \times \mathcal{V} \subseteq H^1_r\times H^1_r $. It is indecisive for the non-radial case.
\end{theorem}

\begin{proof}
The argument follows the proof of Theorem \ref{T:nonradial1}. From Table \ref{T:L2Q}, one can see that the span of $\tilde{Z}$, where $L_2\tilde{Z}=Q$, is negative. Combining with the fact that $\mathrm{Ind}(L^{(0)}_2)=1$, we conclude that $\tilde{Z}$ is the only negative span we have. Therefore, the coercivity of $B_2^{(0)}$ is reached on the subspace orthogonal to $Q$. For the non-radial case, the proof is the same as the orthogonal conditions to extend to the nonradial setting are the same as in the Theorem \ref{T:nonradial1}, and thus, finishing the proof. 
\end{proof}

\begin{remark}
In the dimensions $d=11$ and $d=12$, $\mathrm{ind}(L_2^{(1)})=1$ instead of the zero index, which is not what we obtained in the lower dimensions. We double checked this with the standard 4th order explicit Runge-Kutta method (RK4), taking the step size $h=0.001$ in obtaining solutions to \eqref{index eqn sh}. This led to the same results as in the above calculations via the matlab solver ``\texttt{ode45}". Table \ref{J1}  shows the negative values for $J_{11}^{(1)}$ in $d=11$ and $d=12$, which also suggests that the index of $L_2^{(1)}$ is not zero in those dimensions. We also numerically calculated the negative eigenvalue of $L_2^{(1)}$: $\lambda=-0.01403$ in the dimension $d=11$, and $\lambda=-0.0560$ in the dimension $d=12$. This can be compared with the situation of the index of $L^{(0)}_1$ in the dimensions $d=2$ and $d=3$, $\mathrm{ind}(L_1^{(0)})=1$ when $d=2$ but $\mathrm{ind}(L_1^{(0)})=2$ when $d=3$, see \cite{FMR2006} and \cite{FMR2008}.
\end{remark}

\section{Conclusions}

In this paper we first discussed direct numerical simulations of the generic blow-up solutions in the $L^2$-critical NLS equation in higher dimensions ($d = 4, ..., 12$) under the radial symmetry assumption. Our results show that the ``log-log" law is universal for all $L^2$-critical NLS equations (at least up to $d=12$). 
Secondly, we investigated the Spectral Property 1 in higher dimensions, which is the essential part of the analytical proof of the ``log-log" regime for the cases when the mass of the negative energy initial data is slightly above the mass of $Q$, the corresponding ground state $Q$ for the given dimension and nonlinearity. We confirm that the Spectral Property 1 (as well as a modified version of it) holds from $d=5$ to $d=10$ in a general case, and for $d=11$ and $d=12$, at least, in the radial case. Therefore, we conclude that the ``log-log" blow-up regime is the stable blow-up regime in $d \leq 10$ and radially stable in $d \leq 12$.  

\section*{Appendix A: Justification of ``log-log" corrections in $d = 6, ..., 12$}

Here, we list different functional fittings $F(s)$ for the correction term in the blow-up rate \eqref{FL} in the dimensions $d=6,..., 12$. For dimensions $d=6$ and $7$ we list our computations for $\rho_i$ with $F(s) = 1, \big(\ln \frac1{s} \big)^{\gamma}$, $\gamma = 0.25, 0.2, 0.15$ and $\ln \ln \frac1{s}$ in Tables \ref{loglog comparison 6d} and  \ref{loglog comparison 7d}. 
\begin{table}[h]
\begin{tabular}{|c|c|c|c|c|c|c|}
\hline
\multicolumn{7}{|c|}{The fitting power $\rho_i$ from different corrections $F(s)$}\\
\hline
$i$ &$\frac{1}{L(t)}$ range 
& $1$ 
& $(\ln \frac{1}{s})^{0.25}$
& $(\ln \frac{1}{s})^{0.2}$
& $(\ln \frac{1}{s})^{0.15}$
& $\ln \ln \frac{1}{s}$ \\
\hline
$0$&$2e4 \sim 7e5 $&$0.5104$&$0.5010$&$0.5029$&$0.5048$&$0.4959$\\
\hline
$1$&$2e4 \sim 7e5 $&$0.5064$&$0.5003$&$0.5015$&$0.5027$&$0.4984$\\
\hline
$2$&$7e5 \sim 2e7 $&$0.5045$&$0.5000$&$0.5009$&$0.5018$&$0.4990$\\
\hline
$3$&$2e7 \sim 4e8 $&$0.5035$&$0.4998$&$0.5005$&$0.5013$&$0.4993$\\
\hline
$4$&$4e8 \sim 8e9 $&$0.5028$&$0.4997$&$0.5003$&$0.5009$&$0.4994$\\
\hline
$5$&$8e9 \sim 1e11 $&$0.5024$&$0.4996$&$0.5002$&$0.5007$&$0.4995$\\
\hline
$6$&$1e11 \sim 3e12 $&$0.5020$&$0.4996$&$0.5001$&$0.5006$&$0.4996$\\
\hline
$7$&$3e12 \sim 4e13 $&$0.5018$&$0.4996$&$0.5000$&$0.5005$&$0.4996$\\
\hline
$8$&$4e13 \sim 6e14 $&$0.5016$&$0.4996$&$0.5000$&$0.5004$&$0.4997$\\
\hline
$9$&$6e14 \sim 9e15 $&$0.5014$&$0.4996$&$0.5000$&$0.5003$&$0.4997$\\
\hline
\end{tabular}
\linebreak
\linebreak
\caption{Comparison of the different functional forms $F(s)$ for the correction term in $d=6$. ``$\frac{1}{L(t)}$ range" means the values $\frac{1}{L(t_i)} \sim \frac{1}{L(t_{i+1})}$.}
\label{loglog comparison 6d}
\end{table}
\begin{table}[h]
\begin{tabular}{|c|c|c|c|c|c|c|}
\hline
\multicolumn{7}{|c|}{The fitting power $\rho_i$ from different corrections $F(s)$}\\
\hline
$i$ &$\frac{1}{L(t)}$ range 
& $1$ 
& $(\ln \frac{1}{s})^{0.25}$
& $(\ln \frac{1}{s})^{0.2}$
& $(\ln \frac{1}{s})^{0.15}$
& $\ln \ln \frac{1}{s}$ \\
\hline
$0$&$2e2 \sim 1e4 $&$0.5104$&$0.4995$&$0.5016$&$0.5038$&$0.4924$ \\
\hline
$1$&$1e4 \sim 4e5 $&$0.5063$&$0.4997$&$0.5010$&$0.5023$&$0.4973$ \\
\hline
$2$&$4e5 \sim 9e6 $&$0.5044$&$0.4996$&$0.5006$&$0.5015$&$0.4985$ \\
\hline
$3$&$9e6 \sim 2e8 $&$0.5034$&$0.4996$&$0.5003$&$0.5011$&$0.4990$ \\
\hline
$4$&$2e8 \sim 5e9 $&$0.5028$&$0.4995$&$0.5002$&$0.5008$&$0.4992$ \\
\hline
$5$&$5e9 \sim 1e11 $&$0.5023$&$0.4995$&$0.5001$&$0.5006$&$0.4994$ \\
\hline
$6$&$1e11 \sim 2e12 $&$0.5020$&$0.4995$&$0.5000$&$0.5005$&$0.4995$ \\
\hline
$7$&$2e12 \sim 3e13 $&$0.5017$&$0.4995$&$0.5000$&$0.5004$&$0.4995$ \\
\hline
$8$&$3e13 \sim 5e14 $&$0.5015$&$0.4995$&$0.4999$&$0.5003$&$0.4996$ \\
\hline
$9$&$5e14 \sim 9e15 $&$0.5014$&$0.4995$&$0.4999$&$0.5003$&$0.4996$ \\
\hline
\end{tabular}
\linebreak
\linebreak
\caption{Comparison of the different functional forms $F(s)$ for the correction term in $d=7$. ``$\frac{1}{L(t)}$ range" means the values $\frac{1}{L(t_i)} \sim \frac{1}{L(t_{i+1})}$.}
\label{loglog comparison 7d}
\end{table}
One can observe that log-log shows better stabilization in the approximation of the power $\rho_i$, though, the other approximation seem to reach $0.5$ at some time period and stay there for a while, but then keep decreasing away from it.  

We also give $l^2$ discrepancy $\epsilon_i$ results (for the last three steps from $j_0 = 7$ to $i=9$) in Table \ref{deviation-comparison-67d} for both $d=6,7$. It seems that the smallest error is given by the power $\gamma = 0.2$, which can be explained similarly as in the dimension $d=4$ (during the studied time interval, this functional approximation decreases down to $0.5$, but then it will continue decreasing). Similar comments can be made about other powers of $\gamma$. 
\begin{table}[ht]
\begin{tabular}{|c|c|c|c|c|c|}
\hline
$d$
&$F(s)=1$
&$F(s)=(\ln \frac{1}{s})^{0.25}$  
&$F(s)=(\ln \frac{1}{s})^{0.2}$  
&$F(s)=(\ln \frac{1}{s})^{0.15}$  
&$F(s)=\ln \ln \frac{1}{s}$ \\
\hline
$6$&$0.0017$&$4.04e-4$&$5.56e-5$&$4.51e-4$&$3.71e-4$\\
\hline
$7$&$0.0017$&$4.76e-4$&$6.18e-5$&$3.88e-4$&$4.50e-4$\\
\hline
\end{tabular}
\linebreak
\linebreak
\caption{The $l^2$ discrepancy $\epsilon_i$ calculated from $j_0=7$ to $i=9$ for the fittings given in Tables \ref{loglog comparison 6d} - \ref{loglog comparison 7d} in $d=6$ and $d=7$.} 
\label{deviation-comparison-67d}
\end{table}

For higher dimensions, for brevity, we list $F(s)=1$ and $F(s)= \ln \ln \frac1{s}$ results, see Tables \ref{loglog comparison 8d}, \ref{loglog comparison 9d}, \ref{loglog comparison 10d}, \ref{loglog comparison 11d}, \ref{loglog comparison 12d} for dimensions $d=8, 9, 10, 11, 12$, correspondingly.  
{\small {
\begin{table}[h]
\begin{tabular}{|c|c|c|c|}
\hline
$i$ &$\frac{1}{L(t)}$ range 
& $\rho_i$ from $F(s)=1$ 
& $\rho_i$ from $F(s)=\ln \ln \frac{1}{s}$  \\
\hline
$0$&$7e3 \sim 3e5 $&$0.5069$&$0.4974$ \\
\hline
$1$&$3e5 \sim 8e6 $&$0.5047$&$0.4986$ \\
\hline
$2$&$8e6 \sim 2e8 $&$0.5035$&$0.4991$ \\
\hline
$3$&$2e8 \sim 5e9 $&$0.5028$&$0.4993$ \\
\hline
$4$&$5e9 \sim 1e11 $&$0.5024$&$0.4994$ \\
\hline
$5$&$1e11 \sim 2e12 $&$0.5020$&$0.4995$ \\
\hline
$6$&$2e12 \sim 4e13 $&$0.5017$&$0.4996$ \\
\hline
$7$&$4e13 \sim 8e14 $&$0.5015$&$0.4996$ \\
\hline
$8$&$8e14 \sim 1e16 $&$0.5014$&$0.4996$ \\
\hline
\end{tabular}
\linebreak
\linebreak
\caption{Comparison of the functional fitting in $d=8$.}
\label{loglog comparison 8d}
\end{table}
\begin{table}[ht]
\begin{tabular}{|c|c|c|c|}
\hline
$i$ &$\frac{1}{L(t)}$ range 
& $\rho_i$ from $F(s)=1$ 
& $\rho_i$ from $F(s)=\ln \ln \frac{1}{s}$  \\
\hline
$0$&$4e3 \sim 2e5 $&$0.5069$&$0.4968$ \\
\hline
$1$&$2e5 \sim 7e6 $&$0.5047$&$0.4983$ \\
\hline
$2$&$7e6 \sim 2e8 $&$0.5035$&$0.4989$ \\
\hline
$3$&$2e8 \sim 5e9 $&$0.5028$&$0.4992$ \\
\hline
$4$&$5e9 \sim 1e11 $&$0.5023$&$0.4993$ \\
\hline
$5$&$1e11 \sim 2e12 $&$0.5020$&$0.4994$ \\
\hline
$6$&$2e12 \sim 5e13 $&$0.5017$&$0.4995$ \\
\hline
$7$&$5e13 \sim 9e14 $&$0.5015$&$0.4996$ \\
\hline
$8$&$9e14 \sim 2e16 $&$0.5013$&$0.4996$ \\
\hline
\end{tabular}
\linebreak
\linebreak
\caption{Comparison of the functional fitting in $d=9$.}
\label{loglog comparison 9d}
\end{table}
}}

{\small {
\begin{table}[ht]
\begin{tabular}{|c|c|c|c|}
\hline
$i$ &$\frac{1}{L(t)}$ range 
& $\rho_i$ from $F(s)=1$ 
& $\rho_i$ from $F(s)=\ln \ln \frac{1}{s}$  \\
\hline
$0$&$3e3 \sim 2e5 $&$0.5065$&$0.4959$ \\
\hline
$1$&$2e5 \sim 5e6 $&$0.5044$&$0.4979$ \\
\hline
$2$&$5e6 \sim 2e8 $&$0.5033$&$0.4986$ \\
\hline
$3$&$2e8 \sim 4e9 $&$0.5026$&$0.4990$ \\
\hline
$4$&$4e9 \sim 1e11 $&$0.5022$&$0.4992$ \\
\hline
$5$&$1e11 \sim 2e12 $&$0.5019$&$0.4993$\\
\hline
$6$&$2e12 \sim 5e13 $&$0.5016$&$0.4994$ \\
\hline
$7$&$5e13 \sim 9e14 $&$0.5014$&$0.4995$ \\
\hline
$8$&$9e14 \sim 2e16 $&$0.5013$&$0.4996$ \\
\hline
\end{tabular}
\linebreak
\linebreak
\caption{Comparison of the functional fitting in $d=10$.}
\label{loglog comparison 10d}
\end{table}
\begin{table}[h]
\begin{tabular}{|c|c|c|c|}
\hline
$i$ &$\frac{1}{L(t)}$ range 
& $\rho_i$ from $F(s)=1$ 
& $\rho_i$ from $F(s)=\ln \ln \frac{1}{s}$  \\
\hline
$0$&$2e3 \sim 1e5 $&$0.5060$&$0.4945$ \\
\hline
$1$&$1e5 \sim 4e6 $&$0.5041$&$0.4972$ \\
\hline
$2$&$4e6 \sim 1e8 $&$0.5031$&$0.4982$ \\
\hline
$3$&$1e8 \sim 3e9 $&$0.5025$&$0.4987$ \\
\hline
$4$&$3e9 \sim 7e10 $&$0.5021$&$0.4990$ \\
\hline
$5$&$7e10 \sim 2e12 $&$0.5018$&$0.4992$ \\
\hline
$6$&$2e12 \sim 4e13 $&$0.5015$&$0.4993$ \\
\hline
$7$&$4e13 \sim 8e14 $&$0.5014$&$0.4994$ \\
\hline
$8$&$8e14 \sim 2e16 $&$0.5012$&$0.4995$ \\
\hline
\end{tabular}
\linebreak
\linebreak
\caption{Comparison of the functional fitting in $d=11$.}
\label{loglog comparison 11d}
\end{table}
\begin{table}[h]
\begin{tabular}{|c|c|c|c|}
\hline
$i$ &$\frac{1}{L(t)}$ range 
& $\rho_i$ from $F(s)=1$ 
& $\rho_i$ from $F(s)=\ln \ln \frac{1}{s}$  \\
\hline
$0$&$2e3 \sim 9e4 $&$0.5059$&$0.4940$ \\
\hline
$1$&$9e4 \sim 3e6 $&$0.5040$&$0.4970$ \\
\hline
$2$&$3e6 \sim 1e8 $&$0.5030$&$0.4981$ \\
\hline
$3$&$1e8 \sim 3e9 $&$0.5024$&$0.4987$ \\
\hline
$4$&$3e9 \sim 8e10 $&$0.5020$&$0.4990$ \\
\hline
$5$&$8e10 \sim 2e12 $&$0.5017$&$0.4992$ \\
\hline
$6$&$2e12 \sim 4e13 $&$0.5015$&$0.4993$ \\
\hline
$7$&$4e13 \sim 1e15 $&$0.5014$&$0.4994$ \\
\hline
$8$&$1e15 \sim 2e16 $&$0.5012$&$0.4995$ \\
\hline
\end{tabular}
\linebreak
\linebreak
\caption{Comparison of the functional fitting in $d=12$.} 
\label{loglog comparison 12d}
\end{table}
}}

We conclude this section with providing the $l^2$ discrepancy error (computed for the last three steps) for dimensions $d = 8, ..., 12$ in Table \ref{deviation-comparison-8-12d}. The error $\epsilon_i$ stays on the same order $\sim 10^{-4}$ as in all previous dimensions.
{\small {
\begin{table}[h]
\begin{tabular}{|c|c|c|}
\hline
$d$
&$F(s)=1$
&$F(s)=\ln \ln \frac{1}{s}$ \\
\hline
$8$&$0.0017$&$4.23e-4$\\
\hline
$9$&$0.0016$&$4.69e-4$\\
\hline
$10$&$0.0016$&$5.43e-4$\\
\hline
$11$&$0.0015$&$6.48e-4$\\
\hline
$12$&$0.0014$&$6.87e-4$\\
\hline
\end{tabular}
\linebreak
\linebreak
\caption{The $l^2$ discrepancy $\epsilon_i$ calculated from $j_0=7$ to $i=9$ for the fittings given in Tables \ref{loglog comparison 8d} - \ref{loglog comparison 12d} in $d=8,9,10,11,12$.}
\label{deviation-comparison-8-12d}
\end{table}
}}

\section*{Appendix B: Artificial boundary conditions}\label{app artificial}

The BVP problem (\ref{bvp eqn}) and (\ref{bvp eqn2}) involves the boundary condition $u(\infty)=0$. We use the approach in \cite{MS2011} to construct the artificial boundary condition to approximate $u(\infty)=0$.

As $V_{1,2}$ and $f$ decay exponentially fast, as it was discussed before, for $r\gg 1$ the solutions of (\ref{bvp eqn2}) converge to their linear flows:
\begin{align}
-\partial_{rr}u-\dfrac{d-1+2k}{r}\partial_ru=0, \quad r\gg 1.
\end{align}
This equation has the explicit solution
\begin{align}
u=\dfrac{C_1}{r^{d-2+k}}+C_2:=C_1\phi_1+C_2\phi_2,
\end{align}
where $\phi_1=\frac{1}{r^{d-2+k}}$ and $\phi_2=1$. Since $u(\infty)=0$, only span $\phi_1$ remains. Consequently, at $r=L$, the solution to (\ref{bvp eqn2}) must be linearly dependent on $\phi_1$. Therefore, their Wronskian should equal to $0$, i.e.,
\begin{align} \mathrm{det} \left(
\begin{matrix}
u(L) & \, \phi_1(L) \\
u_r(L) & \, \phi_1'(L)\\
\end{matrix} \right) =0,
\end{align}
which is
\begin{align}\label{abc bvp}
u(L)+\frac{L}{d-2+2k}u_r(L)=0.
\end{align}
Therefore, the boundary condition $u(\infty)=0$ is substituted by $u(L)+\frac{L}{d-2+2k} \, u_r(L)=0$ at the length of the interval $r=L$.
\bigskip

\section*{Appendix C: Computation of the potentials $V_{1}$ and $V_2$}\label{App: Potential}

When $d\geq 5$, the term $Q^{\frac{4}{d}-1}$ in $V_{1,2}$ has the negative power. We rewrite it as follows, for example for $V_2$
\begin{align*}
V_2=\frac{2}{d} \, Q^{\frac{4}{d}} \, r \, \frac{Q_r}{Q}.
\end{align*}

Recall that the ground state $Q$ decays with the rate $Q \sim r^{-\frac{d-1}{2}}e^{-r}$ for $r\gg 1$ (e.g. see \cite{MS2011}). Consequently, $Q_r \sim -Q$ and only differs by a polynomial power, and $V_2$ decays with an exponential rate $V_2 \sim e^{-\frac{4}{d}r}$ for $r \gg 1$.
However, if we compute the $V_2$ from (1.08), instead of keep decreasing, the potential term $V_2$ becomes oscillating at the magnitude $10^{-4}$ when $r \geq 40$ in $d=8$. This fails to describe the ``fast decay" property of the potential terms and may cause trouble in the process of finding the index of $L_{1,2}^{(k)}$ (according to \cite{RS1970}, $L_{1,2}$ has infinitely many negative eigenvalues if the potential decays slower than $\frac{1}{r^2}$). In fact, this issue comes from the numerical calculation of the ground state $Q$. When we compute the ground state $Q$, the $Q$ itself stops decreasing when it reaches the magnitude of $10^{-14}$, since it is less than the tolerance we set for the fixed point iteration and it approaches the machine accuracy. Nevertheless, $(10^{-14})^{\frac{4}{d}}$ could be a relatively large number, especially when the dimension $d$ is large, say $d \geq 8$.

While we cannot find a way to refine the numerical method for finding the ground state $Q$ as it is below our tolerance and approaching the machine accuracy, we come up with an alternative way to compute the potential terms, say $V_2$. We outline it next.

Consider the function
\begin{align}\label{P}
P(r)=Q(r)\, e^{r},
\end{align}
Where $P$ decays with a polynomial rate $r^{-\frac{d-1}{2}}$, since $Q \sim r^{-\frac{d-1}{2}}e^{-r}$ for $r\gg 1$. Then, a straightforward calculation gives us
\begin{align}\label{V2P}
V_2= \frac{2}{d}\, P^{\frac{4}{d}} \, r \, \left( \frac{P_r}{P}-1 \right)e^{-\frac{4}{d}}.
\end{align}
This implies that if we find the profile of $P$, which decays polynomially, then $P$ can not be too small for $r\gg 1$. Moreover, the last term $e^{-\frac{4}{d}}$ describes an exponential decay for the potential $V_2$.

Putting the expression $P(r)=Q(r)e^r$ into the ground state equation (\ref{GS}), we have
\begin{align}\label{GSP}
\begin{cases}
P_{rr} -2P_r + \dfrac{d-1}{r} \left( P_r-P\right) + P^{\frac{4}{d}+1} e^{-\frac{4}{d}r}=0, \qquad (P>0), \\
P_r(0)-P(0)=0, \\
P(\infty)=0.
\end{cases}
\end{align}

We can construct the artificial boundary condition to approximate $P(\infty)=0$ the same way as for $L_{1,2}u=f$. As described in Section \ref{S:SP}, the linear flow of the equation (\ref{GSP}) gives us
\begin{align}
P_{rr} -2P_r + \frac{d-1}{r} \left( P_r-P\right)=0.
\end{align}
The solutions of the free equation decay with the leading order $P \sim r^{-\frac{d-1}{2}}$. Therefore, we can construct the artificial boundary condition at $r=L$:
\begin{align}
\mathrm{det} \left( \begin{matrix}
P(r) & r^{-\frac{d-1}{2}}\\
P_r(r) & -\frac{d-1}{2} \, r^{-\frac{d+1}{2}}
\end{matrix} \right)=0,
\end{align}
i.e.,
\begin{align}
P+\frac{2L}{d-1}P_r=0.
\end{align}

We use the Chebyshev differential matrices to discretize the first and second derivatives $P_r$ and $P_{rr}$, then the discretized equation (\ref{GSP}) changes to the 
nonlinear system
\begin{align}\label{GSPN}
\mathbf{MP}-\mathbf{f(P)}=\mathbf{0},
\end{align}
with the first and last rows substituted by the prescribed boundary conditions. This system \eqref{GSPN} can be solved, for example, by the matlab solver ``\texttt{fsolve}".

Figure \ref{V2f} shows the comparison of the profiles for $V_2$ when $d=8$ and $r\gg 1$.
Table \ref{QPV2} shows the difference of the ground state $Q$ and $Pe^{-r}$, as well as the difference of $V_2$ calculated from $Q$ and $P$, denoted by $V_2$ and $\tilde{V}_2$ correspondingly. Note how this method allows us to describe the ``exponential" decay and completely avoid oscillations!

\begin{figure}[ht]
\begin{center}
\includegraphics[width=0.8\textwidth]{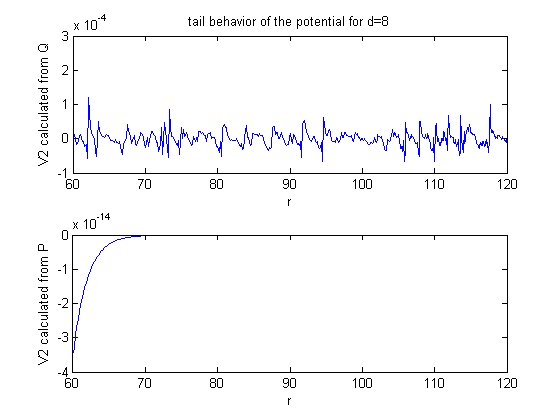}
\caption{Upper: the value of $V_2$ calculated directly from $Q$; lower: $V_2$ obtained from $P$. Our method calculating $V2$ from $P$ completely avoids oscillations.}
\end{center}
\label{V2f}
\end{figure}

\begin{table}
\begin{tabular}{|c|c|c|c|c|c|c|c|c|}
\hline
d&$5$&$6$&$7$&$8$&$9$&$10$&$11$&$12$\\
\hline
$\|Q-Pe^{-r}\|_{\infty}$&$4e-09$&$5e-09$&$7e-09$&$1e-8$&$3e-08$&$4e-7$&$4e-05$&$6e-3$\\
\hline
$\|V_2-\tilde{V}_2\|_{\infty}$&$1e-09$&$2e-07$&$8e-06$&$1e-4$&$1e-3$&$6e-3$& $0.0491$&$0.09$\\
\hline
\end{tabular}
\linebreak
\linebreak
\caption{Comparison of the quantities obtained in two different ways, where $\tilde{V}_2=\frac{2}{d} P^{\frac{4}{d}}r \left( \frac{P_r}{P}-1 \right)e^{-\frac{4}{d}}$. We set $L=100$, $N=1024$.}
\label{QPV2}
\end{table}

\section*{Appendix D: Comparison with prior results on the spectral property in $d=2,3,4$}\label{App: comparison}

In this part we list our numerical results for $d=2,3,4$ to confirm that we get similar results as the ones in Fibich, Merle and Rapha\"{e}l \cite{FMR2006}. One can see that while the different boundary conditions, $u'(L)=0$ and $u(\infty)=0$, generate the different outcomes, this does not affect the proof of the Spectral Property 1 in those dimensions, as the signs of the quantities under consideration remain the same, see Tables \ref{C2d}, \ref{C3d} and \ref{C4d}. We do see some differences 
between our results and the results in \cite{FMR2006}, even when using the same boundary condition (however, this does not influence the outcome for the Spectral Property 1). Note that the situation is different in dimension 5 (and higher) as we discussed in the paper.
\begin{table}[H]
\begin{tabular}{|c|l|l|l|l|l|l|}
\hline
Boundary conditions&$K_{11}$&$K_{22}$&$K_{12}$&$J_{11}$&$J_{22}$&$J_{12}$\\
\hline
In \cite[pp7-8]{FMR2006}, $u'(L)=0$&$-0.65$&$NA$&$NA$&$0.9969$&$12.4211$&$-4.4095$\\
\hline
$u'(L)=0$ from ``cheby"&$-0.65343$&$16.0778$&$-1.3054$&$1.0058$&$12.7804$&$-4.4705$\\
\hline
$u'(L)=0$ from ``\texttt{bvp4c}"&$-0.65344$&$16.0778$&$-1.3054$&$1.0058$&$12.7804$&$-4.4705$\\
\hline
\end{tabular}
\caption{\label{C2d} Comparison of the bilinear forms for $K_{ij}$'s and $J_{ij}$'s by using different boundary conditions. This is for the case $d=2$.}
\end{table}

\begin{table}[H]
\begin{tabular}{|c|l|l|l|l|l|l|}
\hline
Boundary conditions&$K_{11}$&$K_{22}$&$K_{12}$&$J_{11}$&$J_{22}$&$J_{12}$\\
\hline
In \cite[pp7-8]{FMR2006}, $u'(L)=0$&$-2.1$&$NA$&$NA$&$5.9141$&$100.36881$&$-28.0500$\\
\hline
$u'(L)=0$ from ``cheby"&$-2.111$&$-21.058$&$1.6978$&$6.0052$&$104.3914$&$-28.6972$\\
\hline
$u'(L)=0$ from ``\texttt{bvp4c}"&$-2.111$&$-21.058$&$1.6978$&$6.0052$&$104.3913$&$-28.6972$\\
\hline
$u(\infty)=0$&$-2.6868$&$-406.7105$&$16.6004$&$3.0156$&$40.462$&$-14.8724$\\
\hline
\end{tabular}
\caption{\label{C3d} Comparison of the bilinear forms for $K_{ij}$'s and $J_{ij}$'s by using different boundary conditions. This is for the case $d=3$.}
\end{table}

\begin{table}[H]
\begin{tabular}{|c|l|l|l|l|l|l|}
\hline
Boundary conditions&$K_{11}$&$K_{22}$&$K_{12}$&$J_{11}$&$J_{22}$&$J_{12}$\\
\hline
In \cite[pp7-8]{FMR2006}, $u'(L)=0$&$-8.0178$&$-31.8102$&$4.6472$&$147$&$3404$&$-728$\\
\hline
$u'(L)=0$ from ``cheby"&$-8.0169$&$-32.0677$&$4.6791$&$158.4112$&$3764.4032$&$-792.1838$\\
\hline
$u'(L)=0$ from ``\texttt{bvp4c}"&$-8.0169$&$-32.0677$&$4.6791$&$158.4114$&$3764.4103$&$-792.1853$\\
\hline
$u(\infty)=0$&$-8.0401$&$-159.7802$&$2.9583$&$13.4391$&$200.4192$&$-73.3805$\\
\hline
\end{tabular}
\caption{\label{C4d} Comparison of the bilinear forms for $K_{ij}$'s and $J_{ij}$'s by using different boundary conditions. This is for the case $d=4$.}
\end{table}

To further justify our numerical results, we also studied the ($L^2$-supercritical) 3d cubic NLS equation ($d=3$ and $p=3$), which was discussed by Marzuola and Simpson in \cite{MS2011}. From Table \ref{Compare MS}, one can see that our numerical results match the results of Marzuola and Simpson \cite{MS2011} very well.
\begin{table}[H]
\begin{tabular}{|c|c|c|}
\hline
3d cubic case &$K_{11}$&$J_{11}$\\
\hline
Numerical results in \cite[pp14]{MS2011}&$1.04846$&$-0.662038$\\
\hline
Our numerical results by ``cheby"&$1.04846$&$-0.662038$\\
\hline
\end{tabular}
\caption{\label{Compare MS} Comparison of the bilinear forms for $K_{11}$'s and $J_{11}$'s with the results from \cite{MS2011} obtained by ``\texttt{bvp4c}", here, $L_2Z_1=Q+rQ_r$ and $J_{11}=\langle  L_2Z_1, Z_1 \rangle$. This is for the 3d cubic NLS equation ($L^2$-supercritical).}
\end{table}


\begin{thebibliography}{00}

\bibitem{ADKM2003}
G. D. Akrivis, V. A. Dougalis, O. A. Karakashian and W. R. McKinney,
\textit{Numerical approximation of
blow-up of radially symmetric solutions of the nonlinear Schr\"{o}dinger equation},
SIAM J. Sci. Comput. 25 (2003), no. 1, 186–-212.

\bibitem{ADKM1993}
G. D. Akrivis, V. A. Dougalis, O. A. Karakashian and W. R. McKinney,
\textit{Galerkin-finite element methods
for the nonlinear Schr\"{o}dinger equation},
in ``Advances on Computer Mathematics and its Applications" (E. Lipitakis ed.),
World Scientific, Singapore 1993, 85–-106.

\bibitem{BK2015}
M. Birem and C. Klein,
\textit{Multidomain spectral method for Schr\"{o}dinger equations},
{Adv. Comput. Math. (2016) 42, 395--423.}


\bibitem{BL1983}
H. Berestycki and P. Lions,
\textit{Nonlinear scalar field equations. I. Existence of a ground state},
{Arch. Ration. Mech. Anal. 82(4)(1983), 313--345.}

\bibitem{BZS1975}
V. Budneva, V. Synakh, V. Zakharov,
\textit{Certain modes for wave collapse},
{Sov. J. Plasma Phys. 1. 1975, 335--338.}

\bibitem{BP}
V. S. Buslaev and G. S. Perelman,
\textit{Nonlinear scattering: the states which are close to a soliton},
Zap. Nauchn. Sem. POMI, 1992, vol. 200, 38--50.

\bibitem{Ca2003}
T. Cazenave,
\textit{Semilinear Schr\"{o}dinger equations},
{American Mathematical Soc, 2003.}


\bibitem{DNPZ1992}
S. Dyachenko, A. Newell, V. Zakharov,
\textit{Optical turbulence: weak turbulence, condensates and collapsing filaments in the nonlinear Schr\"{o}dinger equation},
{Physica D 57, 1992, 96--160.}


\bibitem{F2016}
G. Fibich,
\textit{The nonlinear Schr\"{o}dinger equation, singular solutions and optical collapse},
{Springer, 2015.}

\bibitem{FGW2005}
G. Fibich, N. Gavish, X. Wang,
\textit{New singular solutions of the nonlinear Schr\"{o}dinger equation},
{Physica D, 211 (2005), 193--220.}

\bibitem{FMR2006}
G. Fibich, F. Merle, P. Rapha\"{e}l,
\textit{Proof of a spectral property related to the singularity formation for the $L^2$ critical nonlinear Schr\"{o}dinger equation},
{Physica D, 220: 2006, 1--13.}

\bibitem{FMR2008}
G. Fibich, F. Merle, P. Rapha\"{e}l,
\textit{Erratum to ``Proof of a spectral property related to the singularity formation for the $L^2$ critical nonlinear Schr\"{o}dinger equation"},
{personal communication.}

\bibitem{FP1999}
G. Fibich and G. Papanicolaou,
\textit{Self-focusing in the perturbed and unperturbed nonlinear Schr\"{o}dinger equation in critical dimension}, 
SIAM J. Appl. Math., 60 (1999), 183-–240.

\bibitem{F85}
G. M. Fraiman,
\textit{Asymptotic stability of manifold of self-similar solutions in self-focusing},
Soviet Phys. JETP 61 (1985), no. 2, 228-–233; translated from
Zh. \'Eksper. Teoret. Fiz. 88 (1985), no. 2, 390--400 (Russian).

\bibitem{GV1979}
J. Ginibre, G. Velo,
\textit{On a class of nonlinear Schr\"{o}dinger equations I. The Cauchy problem, general case},
{J. Funct. Anal. 32(1)(1979), 1--32.}

\bibitem{G1977}
R. Glassey,
\textit{On the blowing up of solutions to the Cauchy problem for nonlinear Schr\"{o}dinger equation},
{J. Math. Phys. 18 (1977), no. 9,  1794--1797.}

\bibitem{SZ}
M. Goldman, K. Rypdal, B. Hafizi,
\textit{Dimensionality and dissipation in Langmuir collapse},
{Phys. Fluids 23, 1980, 945--955.}

\bibitem{C2014}
C. Guevara,
\textit{Global behavior of finite energy solutions to the focusing Nonlinear Schr\"{o}dinger Equation in $d$-dimensions.}
{Appl Math Res Express (2014) vol.2017, 177--243.}


\bibitem{HRP2010}
J. Holmer, R. Platte and S. Roudenko,
\textit{Blow-up criteria for the 3d cubic NLS equation},
{Nonlinearity, 23 (2010), 977--1030.}

\bibitem{HR2012}
J. Holmer and S. Roudenko,
\textit{Blow-up solutions on a sphere for the 3d quintic NLS in the energy space}, 
Analysis \& PDE, 5-3 (2012), 475--512.

\bibitem{HR2011}
J. Holmer and S. Roudenko,
\textit{A class of solutions to the 3d cubic nonlinear Schr\"{o}dinger equation that blow-up on a circle},
AMRX Appl Math Res Express (2011), vol. 2011, 23--94. doi:10.1093/amrx/abq016

\bibitem{HR2007}
J. Holmer and S. Roudenko, 
\textit{On blow-up solutions to the 3d cubic nonlinear Schr\"{o}dinger equation}, 
AMRX Appl. Math. Res. Express, 1 (2007), article ID abm004, 31 pp.,
doi:10.1093/amrx/abm004

\bibitem{KSZ1991}
N. E. Kosmatov, V. F. Shvets, V. E. Zakharov,
\textit{Computer simulation of wave collapses in the nonlinear Schr\"{o}dinger equation},
{Physica D 52 (1991), 16--35.}

\bibitem{Kwong1989}
M.K. Kwong,
\textit{Uniqueness of positive solutions of $\Delta u-u + u^p = 0$ in $\mathbb{R}^n$},
{Arch. Ration. Mech. Anal. 105 (3) (1989), 243--266.}

\bibitem{LLPSS1989}
M. J. Landman, B. LeMesurier, G. Papanicolaou, C. Sulem, P.-L. Sulem,
\textit{Singular solutions of the cubic Schr\"{o}dinger equation, ``integrable systems and applications"},
{Lecture Notes in Physics, 342 (2005), Springer-Verlag 207--217.}


\bibitem{LPSS1988}
M. J. Landman, G. Papanicolaou, C. Sulem, P.-L. Sulem,
\textit{Rate of blowup for solutions of the nonlinear Schr\"{o}dinger equation at critical dimension},
{Phys. Rev. A(3) 38(8) (1988), 3837--3843.}

\bibitem{LPSSW1990}
M. J. Landman, G. Papanicolaou, C. Sulem, P.-L. Sulem, X. Wang,
\textit{Stability of isotropic singularities for the nonlinear Schr\"{o}dinger equation},
{Phys. Rev. A(3) 38(8) (1988), 3837--3843.}


\bibitem{LePSS1988}
B. LeMesurier, G. Papanicolaou, C. Sulem, P.-L. Sulem,
\textit{Local structure of the self-focusing singularity of the nonlinear Schr\"{o}dinger equation},
{Physica D 32 (1988), 210--226.}

\bibitem{LDV2013}
P. M. Lushnikov, S. A. Dyachenko and N. Vladimirova,
\textit{Beyond leading-order logarithmic scaling in the catastrophic self-focusing of a laser beam in Kerr media},
{Physical Review A, v. 88 (2013), 13--845.}

\bibitem{M1993}
V. M. Malkin, 
\textit{On the analytical theory for stationary self-focusing of radiation},
Physica D, 64 (1993), 251--266.

\bibitem{MS2011}
J. Marzuola, G. Simpson,
\textit{Spectral Analysis for Matrix Hamiltonian Operators},
{Nonlinearity, vol. 24 (2011), 389--429.}

\bibitem{MPSS1986}
D. McLaughlin, G. Papanicolaou, C. Sulem, P.-L. Sulem,
\textit{Focusing singularity of the cubic Schr\"{o}dinger equation},
{Physical Review A, vol. 34 (1986) 1200--1210.}

\bibitem{M-CPAM}
F. Merle,
\textit{On uniqueness and continuation properties after blow-up time of self-similar solutions of nonlinear Schr\"odinger equation with critical exponent and critical mass}, Comm. Pure Appl. Math. 45 (1992), no. 2, 203-–254.

\bibitem{M-Duke}
F. Merle,
\textit{Determination of blow-up solutions with minimal mass for nonlinear Schr\"odinger equations with critical power},
Duke Math. J. 69 (1993), no. 2, 427-–454.


\bibitem{MR2005}
F. Merle, P. Rapha\"{e}l,
\textit{Blow up dynamic and upper bound on the blow up rate for critical nonlinear Schr\"{o}dinger equation},
{Ann. Math. 161 (2005), no. 1, 157--222.}

\bibitem{MR2003}
F. Merle, P. Rapha\"{e}l,
\textit{Sharp upper bound on the blow up rate for critical nonlinear Schr\"{o}dinger equation},
{Geom. Funct. Anal. 13 (2003), 591--642.}

\bibitem{MR2004}
F. Merle, P. Rapha\"{e}l,
\textit{On universality of blow up profile for $L^2$ critical nonlinear Schr\"{o}dinger equation},
{Invent. Math. 156 (2004), 565--672.}

\bibitem{MR2006}
F. Merle, P. Rapha\"{e}l,
\textit{Sharp lower bound on the blow up rate for critical nonlinear Schr\"{o}dinger equation},
{J. Amer. Math. Soc. 19 (2006), no. 1, 37--90.}

\bibitem{GP2001}
Galina Perelman,
\textit{On the blow up phenomenon for the critical nonlinear Schr\"{o}dinger equation in 1D},
{Ann. Henri Poincar\'{e} 2(2001), 605--673.}

\bibitem{GP1999}
Galina Perelman,
\textit{On the blow up phenomenon for the critical nonlinear Schr\"{o}dinger equation in 1D}, S\'{e}minaire: \'{E}quations aux D\'{e}riv\'{e}es Partielles, 1999–2000, Exp. No. III, 16 pp., \'{E}cole Polytech., Palaiseau, 2000. 

\bibitem{R2005}
P. Rapha\"{e}l,
\textit{Stability of the log-log bound for blow up solutions to the critical nonlinear Schr\"{o}dinger equation},
{Math. Ann. 331 (2005), no. 3, 577--609.}


\bibitem{RS1970}
M. Reed, B. Simon,
\textit{Method of modern mathematical physics IV. Analysis of Operators.}
{Academic Press 1978.}

\bibitem{RR1986}
K. Rypdal, J. Rasmussen,
\textit{Blow-up in nonlinear Schr\"{o}dinger equations II},
{Phys. Scripta 33, 1986, 498--504.}

\bibitem{SZ2011}
G. Simpson, I. Zwiers,
\textit{Vortex collapse for the $L^2$-critical nonlinear Schr\"{o}dinger equation},
{J. Math. Phys., 52(8), (2011), 83--503.}

\bibitem{STT2015}
J. Shen, T. Tang, L. Wang,
\textit{Spectral Method, algorithms, analysis and applications.}
{Springer 2015.}


\bibitem{SS1999}
C. Sulem, P.-L. Sulem,
\textit{The nonlinear Schr\"{o}dinger equation. Self-focusing and wave-collapse.}
{Springer 1999.}

\bibitem{SZ2}
P.-L. Sulem, C. Sulem, A. Patera,
\textit{Numerical simulation of singular solutions of the two-dimensional cubic Schr\"{o}dinger equation},
{Com. Pure and Appl. Math, v. 37 (1984), 755--778.}

\bibitem{ZS1976}
V. Synakh, V. Zakharov,
\textit{The nature of the self-focusing singularity},
{Sov. Phys. JETP 41 (1975), 465--468.}

\bibitem{T}
E. C. Titchmarsh, 
\textit{Eigenfunction expansions associated with second-order differential equations}, 
Oxford, Clarendon Press, 1946.

\bibitem{T2001}
L. N. Trefethen,
\textit{Spectral Method in Matlab.}
{SIAM 2000.}

\bibitem{VPT1971}
S. N. Vlasov, V. A. Petrishchev, and  V. I. Talanov,
\textit{Averaged  description  of wave  beams in linear and nonlinear  media (the method  of moments)},
Radiophysics and Quantum Electronics 14 (1971) pp. 1062--1070.
Translated from Izvestiya Vysshikh Uchebnykh Zavedenii,
Radiofizika, 14 (1971) pp. 1353--1363.

\bibitem{VPT1978}
S. N. Vlasov, L. V. Piskunova, V. I. Talanov,
\textit{Structure of the field near a signularity arising from self-focusing in a cubically nonlinear medium},
{Sov. Phys. JETP 48. 1978, 808--812.}


\bibitem{W1983}
M. I. Weinstein,
\textit{Nonlinear Schr\"{o}dinger equations and sharp interpolation estimates},
{Comm. Math. Phys. 87 (1983) 567--576.}

\bibitem{W1985}
M. I. Weinstein,
\textit{Modulational stability of ground states of nonlinear Schr\"{o}dinger equations},
{SIAM J. Anal. 16 (1985) 472--491.}

\bibitem{W1986}
M. I. Weinstein,
\textit{Lyapunov stability of ground states of nonlinear dispersive evolution equations},
{Comm. Pure Appl. Math., Volume 39 (1986) 51--68.}

\bibitem{W1984}
D. Wood,
\textit{The self-focusing singularity in the nonlinear Schr\"{o}dinger equation},
{Stud. Appl. Math. 71  (1984), 103--115.}

\bibitem{Kai-Thesis}
Kai Yang, 
\textit{Formation of singularities in nonlinear dispersive PDE}, Pro-Quest LLC, 2018, Thesis (PhD) -- The George Washington University; arXiv:1712.07647

\bibitem{Z1972}
V. E. Zakharov,
\textit{Collapse of Langmuir waves},
Zh. Eksp. Teor. Fiz. 62, (1972), 1745--1751, (in Russian);
Sov. Phys. JETP, 35 (1972), 908--914 (English).



\end{thebibliography}
\end{document}